\documentclass[11pt,a4paper]{article}

\usepackage{tikz,tkz-euclide,tikz-cd}
\usepackage{amsmath,amssymb,amsfonts,amsthm,amsrefs}

\usepackage{graphicx}
\usepackage[colorlinks=true,
            linkcolor=blue,
            citecolor=blue,
            urlcolor=blue,
            pdfborder={0 0 0}]{hyperref}
\usepackage{geometry}
\usepackage{booktabs,multirow}
\usepackage[ruled,vlined]{algorithm2e}

\usepackage{array, longtable}
\usepackage{enumerate}
\usepackage{todonotes}
\usepackage{xypic}
\usepackage{mathtools}
\usepackage{enumitem}

\newtheorem{introthm}{Theorem}

\newtheorem{theorem}{Theorem}[section]
\newtheorem{proposition}[theorem]{Proposition}
\newtheorem{lemma}[theorem]{Lemma}

\newtheorem{conjecture}[theorem]{Conjecture}

\theoremstyle{definition}

\newtheorem{example}[theorem]{Example}

\theoremstyle{remark}
\newtheorem{remark}[theorem]{Remark}

\numberwithin{equation}{section}

\begin{document}

\title{On Cox Rings of Calabi--Yau hypersurfaces}
\author{Michela Artebani, Antonio Laface, and Luca Ugaglia}
\date{\today}

\maketitle


\begin{abstract}
We study the Cox rings of smooth anticanonical Calabi–Yau hypersurfaces in smooth toric Fano varieties. Using the combinatorics of primitive pairs of the ambient Fano polytope~\cite{bat} and the description of Cox rings of embedded varieties via localizations~\cite{hlu}, we identify several configurations for which the hypersurface is a Mori dream space and obtain explicit presentations of its Cox ring. We also exhibit combinatorial configurations forcing the birational automorphism group to be infinite, yielding in dimensions three and four a dichotomy between finite generation of the Cox ring and infinite birational automorphism group. Finally, for a class of non–Mori dream examples, we prove the Morrison–Kawamata cone conjecture for the movable cone.
\end{abstract}

\maketitle

\begingroup
\renewcommand{\thefootnote}{}
\footnotetext{
\textit{2020 Mathematics Subject Classification.}
Primary 14C20, 14M25; Secondary 14J32, 14E30, 14E07, 14J28, 14Q15.

\textit{Key words and phrases.}
Cox rings, Mori dream spaces, Calabi--Yau hypersurfaces, toric Fano varieties, primitive collections, birational automorphisms, movable cone, Morrison--Kawamata cone conjecture, K3 surfaces.
}
\endgroup
\setcounter{footnote}{0}

\section*{Introduction}
The Cox ring of a projective variety encodes in a single multigraded algebra the spaces of sections of all divisor classes, and its finite generation has strong consequences for birational geometry. By the foundational work of Hu and Keel, finite generation of the Cox ring is equivalent, under mild hypotheses, to the variety being a Mori dream space; in that case the effective cone admits a finite chamber decomposition, every divisor runs a Mori program, and the birational models of the variety are controlled by variation of GIT quotients \cite{hk}. From this perspective, Cox rings provide a natural algebraic framework in which to study the birational geometry of higher-dimensional varieties.

A broad and remarkably well behaved class is given by varieties of Fano type. By the results of Birkar--Cascini--Hacon--McKernan, varieties of Fano type are Mori dream spaces, and hence have finitely generated Cox rings \cite{bchm}. In particular, smooth toric Fano varieties, whose Cox rings are polynomial rings, furnish the basic testing ground for explicit computations. This suggests a natural next problem: to understand how much of this picture survives when one passes from the Fano world to the Calabi--Yau world. From the point of view of Cox rings, Calabi--Yau varieties are among the first classes for which finite generation becomes genuinely subtle and cannot be expected in general.

This transition is also natural from the perspective of the minimal model program. As emphasized in McKernan's survey on Mori dream spaces \cite{mk}, finite generation should be seen as a global form of birational control: it packages the outputs of the MMP for all divisors into finitely many chambers. For Calabi--Yau varieties, however, birational geometry is expected to be governed not by the rational polyhedrality of the nef or movable cones themselves, but rather by the action of automorphism and birational automorphism groups on those cones. This is the philosophy of the Morrison--Kawamata cone conjecture \cites{mor,kaw, lop}. Thus, from the Cox ring viewpoint, Calabi--Yau varieties are special precisely because they lie at the boundary between finite generation and infinite birational symmetry.

In this paper we study this problem for smooth anticanonical hypersurfaces in smooth toric Fano varieties. Let $\Delta$ be a smooth Fano polytope, let $\mathbb P_{\Delta}$ be the associated smooth toric Fano variety, and let
\[
X \in |-K_{\mathbb P_{\Delta}}|
\]
be a general anticanonical hypersurface (very general if $\dim(X)=2$). Then $X$ is a smooth Calabi--Yau variety. Our goal is to relate the finite generation of the Cox ring of $X$ to the combinatorics of the  polytope $\Delta$, and more precisely to the primitive pairs of $\Delta$ \cite{bat}.

The starting point is the general method developed in \cite{hlu}, which describes the Cox ring of an embedded variety as an intersection of localizations of the Cox ring of the ambient Mori dream space. In the present setting, the codimension two components of the irrelevant locus of $\mathbb P_{\Delta}$ are in bijection with primitive pairs of $\Delta$. A first key observation is that the multiplicity with which a general anticanonical equation vanishes along such a codimension two stratum is exactly the degree of the corresponding primitive pair. This allows one to read the structure of the Cox ring of $X$ directly from the primitive relations of $\Delta$.

Our first main result identifies several configurations of primitive pairs for which $X$ is a Mori dream space and gives an explicit presentation of the Cox ring.

\begin{introthm}
\label{thm:1}
Let $\Delta$ be a smooth Fano polytope of dimension $n\geq 3$, 
with vertices
$v_1,\dots,v_r$, and let $\mathbb{P}_\Delta$ be the corresponding toric variety. 
Fix a general anticanonical hypersurface
(very general if $n=3$)
\[
X=\{f(x_1,\dots,x_r)=0\}\subseteq \mathbb{P}_\Delta,
\]
where $x_i$ is the generator of the Cox ring corresponding to the vertex $v_i$.  
If the primitive pairs of 
$\Delta$ are exactly those appearing
in the following table, then 
$X$ is a Mori dream space and its
Cox ring is as follows.

{\footnotesize
\begin{center}
\begin{tabular}{c|c|c|c}
\toprule
& Relations & Equation & $\mathcal R(X)$ \\ 
\midrule
(i) & $v_1+v_2=v_3$
& $f_1x_1+f_2x_2$ &
$\displaystyle
\frac{\mathbb{K}[x_1,\dots,x_r,s]}
     {\langle\,f_2 - s\,x_1,\;f_1 + s\,x_2\,\rangle}
$ 
\\[15pt]
(ii) & $v_1+v_2=0$
& $f_1x_1^2+f_2x_1x_2+f_3x_2^2$&
$\displaystyle
\frac{\mathbb{K}[x_1,\dots,x_r,s_1,s_2]}
     {\langle\,f_3 - x_1\,s_1,\;f_2 + x_1\,s_2 + x_2\,s_1,\;f_1 - x_2\,s_2\,\rangle}
$ 
\\[15pt]
(iii) & $\begin{array}{l}
 v_1+v_2 = 0,\\   
 v_1+v_4 =v_3\\
 v_2+v_3 = v_4
 \end{array}$
& $f_1\,x_1^{2}x_3 + f_2\,x_1x_2 + f_3\,x_2^{2}x_4$ &
$\displaystyle
\frac{\mathbb{K}[x_1,\dots,x_r,s_1,s_2]}
      {\langle f_3-x_1s_1,\ 
      f_2+x_1x_3s_2+x_2x_4s_1,\ 
      f_1-x_2s_2\rangle}$
\\[20pt]
(iv) & 
$\begin{array}{l}
v_1+v_2=v_5,\\ 
v_3+v_4 = v_6
\end{array}$
&
$f_1\,x_1x_3 + f_2\,x_1x_4 + f_3\,x_2x_3 + f_4\,x_2x_4$
&
$\displaystyle
\frac{\mathbb{K}[x_1,\dots,x_r,s_1,s_2]}
{\left\langle
\begin{array}{c}
f_1f_4 - f_2f_3 + s_1s_2,\ 
    x_2s_2 + x_3f_1 + x_4f_2,\\[2pt]    x_1s_2 - x_3f_3 - x_4f_4,\ 
    x_1f_2 + x_2f_4 - x_3s_1,\\[2pt]
    x_1f_1 + x_2f_3 + x_4s_1
\end{array}
\right\rangle}$
\\[15pt]
\bottomrule
\end{tabular}
\end{center}
}
\end{introthm}

Cases $(i)$ and $(ii)$ of  
Theorem~\ref{thm:1} generalize \cite[Cor.~5.1]{hlu}, which computes the Cox ring of a general anticanonical Calabi–Yau hypersurface in a smooth toric Fano variety with Picard rank~$2$. In that case, the irrelevant locus of the ambient toric variety has at most one component of codimension two.

Our second main result identifies configurations which force the opposite behaviour.

\begin{introthm}\label{thm:2}
Let $\Delta$ be a smooth Fano polytope of dimension $n\geq 3$, let $\mathbb{P}_\Delta$ be the corresponding toric variety, and let $X \subseteq \mathbb{P}_\Delta$ be a smooth general anticanonical hypersurface (very general if $n=3$). 
Assume that $\Delta$ contains four distinct vertices $v_1,v_2,v_3,v_4$ such
that one of the following holds:
\begin{itemize}
    \item[i)] $v_1 + v_2 = v_3 + v_4 = 0$, or
    \item[ii)] $v_1 + v_2 = 0$ and  $v_3 + v_4 = v_1$.
\end{itemize}
Then the image of $\mathrm{Bir}(X)$ in 
\(\operatorname{GL}(N^1(X)_\mathbb R)\) is infinite. In particular, \(X\) is not a Mori dream space.
\end{introthm}

Case $i)$ of Theorem~\ref{thm:2} extends the known non--Mori dream space phenomena for Calabi--Yau hypersurfaces in products of projective spaces.
Indeed, let $\mathbb{P}=(\mathbb{P}^1)^m\times \mathbb{P}^{n_1}\times\cdots\times \mathbb{P}^{n_k}$
with $n_j\ge 2$, and let $\Delta$ be the smooth Fano polytope associated with $\mathbb{P}$.
For each $\mathbb{P}^1$-factor, $\Delta$ contains a primitive pair of degree~$2$; hence if $m\ge 2$
there are at least two such pairs and Theorem~\ref{thm:2} applies to a (very) general
anticanonical hypersurface $X\subset \mathbb{P}$, yielding $\mathrm{Bir}(X)$ is infinite.
This recovers the source of the infinite birational symmetry appearing in Kawamata's example 
$\mathbb{P}=\mathbb{P}^1\times\mathbb{P}^1\times\mathbb{P}^2$ and is consistent with Ottem's result
that when $m>1$ the group of pseudo-automorphisms is infinite
and the movable cone is not rational polyhedral \cite{ott}. From the perspective discussed above, this is precisely the type of phenomenon predicted by the Morrison--Kawamata philosophy: the failure of finite generation is accompanied by the presence of large birational symmetry.  

For Calabi--Yau hypersurfaces in smooth toric Fano varieties of ambient dimensions $3$ and $4$, our results yield a dichotomy between finite generation of the Cox ring and infinite birational automorphism group. 
We make this classification explicit using the standard indexing of the \emph{Graded Ring Database} 
(\url{http://www.grdb.co.uk/forms/toricsmooth}). 
The cases not covered directly by Theorems~\ref{thm:1} and~\ref{thm:2} are settled by two additional arguments. 
Most of them are handled by the auxiliary algorithm \textsc{TestFace} (Alg.~\ref{alg}), which certifies the polyhedrality of the effective cone by testing its facets on suitable surface sections via the Hodge index theorem. 
The remaining exceptional case is treated by an explicit Cox ring computation in Example~\ref{ex:117}. 

\begin{introthm}
\label{cor:1}
Let $\Delta$ be a smooth Fano polytope of dimension $n\in\{3,4\}$, let 
$\mathbb{P}_\Delta$ be the associated toric variety, and let 
$X\subseteq\mathbb{P}_\Delta$ be a smooth general anticanonical hypersurface, 
very general when $n=3$. Then either $X$ is a Mori dream space, or 
$\Delta$ contains four distinct vertices $v_1,v_2,v_3,v_4$ appearing in one of 
the two configurations described in Theorem~\ref{thm:2}; in this latter case the 
birational automorphism group $\mathrm{Bir}(X)$ is infinite.
The following table records whether the Calabi--Yau hypersurface associated 
with each index is a Mori dream space.
\begin{table}[h]
\centering
\small
\renewcommand{\arraystretch}{1.2}
\begin{tabular}{@{} l l p{2.5cm} p{7.5cm} @{}}
\toprule
\textbf{MDS} & \textbf{Method} & \textbf{Dimension 3} & \textbf{Dimension 4} \\
\midrule

\multirow{7}{*}{Yes}
 & No primitive pairs 
   & 23 
   & 44, 70, 141, 146, 147 \\[0.2em]

 & Thm.~\ref{thm:1} (i)
   & 19 
   & 40, 60, 69, 137, 138 \\[0.2em]

 & Thm.~\ref{thm:1} (ii)
   & 7, 20, 22 
   & 25, 41, 64, 139, 144, 145 \\[0.2em]

 & Thm.~\ref{thm:1} (iii)
   & 6, 16 
   & 24, 36, 65, 127, 128 \\[0.2em]

 & Thm.~\ref{thm:1} (iv)
   & 
   & 53, 55, 97, 109 \\[0.2em]

 & Alg.~\ref{alg}
   & 
   & 33, 34, 35, 38, 54, 93, 94, 104, 110, 132, 133 \\[0.2em]

 & Ex.~\ref{ex:117}
   & 
   & 117 \\

\midrule

\multirow{6}{*}{No}
 & Thm.~\ref{thm:2} (i)
   & 21
   & \begingroup\footnotesize
     37, 42, 52, 59, 61, 62, 63, 66, 126, 140, 142, 143
     \endgroup \\[0.6em]

 & Thm.~\ref{thm:2} (ii)
   & 11, 12, 18
   & \begingroup\footnotesize
     29, 30, 31, 39, 43, 47, 49, 50, 58, 68, 74, 75, 83, 96, 105, 108, 111, 114, 115, 116, 129, 134, 135, 136
     \endgroup \\[0.6em]
 & Thm.~\ref{thm:2} (i) and (ii)
   & 8, 9, 10, 13, 14, 15, 17
   & \begingroup\footnotesize
     26, 27, 28, 32, 45, 46, 48, 51, 56, 57, 67, 71, 72, 73, 76, 77, 78, 79, 80, 81, 82, 84, 85, 86, 87, 88, 
89, 90, 91, 92, 95, 98, 99, 100, 101, 102, 103, 106, 107, 112, 113, 118, 119, 120, 121, 122, 123, 124, 125, 130, 131
     \endgroup \\

\bottomrule
\end{tabular}
\caption{Hypersurfaces in smooth Fano toric varieties of dimensions $3$ and $4$.}
\label{tab:34}
\end{table}
\end{introthm}

In the final part of the paper, we explore the validity of the Morrison--Kawamata cone conjecture for the movable cone of some non-Mori dream Calabi--Yau hypersurfaces in our setting. 
Proposition~\ref{prop:cone-conjecture-from-two-involutions} shows that the Morrison--Kawamata cone conjecture holds if $X$ carries two pseudo-automorphisms of order two whose reflecting walls are two facets of ${\rm Nef}(X)$, whose composition is unipotent and fixes a common nef ray, and if all remaining facets of ${\rm Nef}(X)$ already belong to $\partial{\rm Mov}(X)$. Under these hypotheses ${\rm Mov}(X)$ is tiled by translates of ${\rm Nef}(X)$ under the group generated by the two involutions. In particular, this yields a rational polyhedral fundamental domain for the action of ${\rm Bir}(X)$ on the movable cone.
 A key point here is that, 
 via the identification $N^1(Z)_{\mathbb R}\simeq N^1(X)_{\mathbb R}$, ${\rm Nef}(Z)={\rm Nef}(X)$ \cite{bor},
 and one can  determine precisely which facets of such cone contribute to the boundary of the movable cone of $X$ (see Proposition \ref{prop:facet-nef-to-mov}).
 As a consequence we can prove the following.
 \begin{introthm}\label{thm:cone}
Let $\mathbb P_\Delta$ be a smooth toric Fano variety of dimension $n\geq 4$
isomorphic to
\[
\mathbb P^1\times \mathbb P_{\mathbb P^{n-2}}(\mathcal O\oplus\mathcal O(i)),
\]
where $0\leq i\leq n-2$, and let $X\in |-K_{\mathbb P_\Delta}|$ be a smooth general member.
Then the Morrison--Kawamata cone conjecture holds for $X$.
\end{introthm}

The paper is organized as follows. In Section~\ref{sec:1} we recall the necessary background on smooth Fano polytopes, primitive collections, and Cox rings, and we relate primitive pairs to codimension-two components of the irrelevant locus. Section~\ref{sec:2} reviews the embedded-variety method for computing Cox rings. In Section~\ref{sec:3} we prove Theorem~\ref{thm:1} and obtain explicit presentations in the Mori dream cases. Section~\ref{sec:4} is devoted to the proof of Theorem~\ref{thm:2}. In Section~\ref{sec:5} we establish the low-dimensional classification stated in Theorem~\ref{cor:1}. Section~\ref{sec:6} contains the proof of Theorem~\ref{thm:cone}, while Section~\ref{sec:7} discusses the K3 case. The appendix collects the data needed to carry out the argument used in the proof of Theorem~\ref{cor:1}.


\section{Preliminaries}
\label{sec:1}

This section collects the toric and combinatorial preliminaries needed in the sequel. 
We recall the notions of reflexive, terminal and smooth Fano polytopes, together with the associated toric Fano varieties and their basic numerical invariants. 
We then review the description of the Cox ring and of the irrelevant locus of a toric variety in terms of the vertices and faces of the defining polytope. 
Special attention is given to codimension-$2$ components of the irrelevant locus, which correspond to primitive pairs. 
For such pairs, we relate the degree of the primitive relation to the vanishing multiplicity of anticanonical sections along the corresponding component of the irrelevant locus. 
This relation will be used later to distinguish the two possible cases, namely primitive pairs of degree~$1$ and degree~$2$. 
We also record a reformulation of the degree-$2$ case in terms of lattice width, and conclude with a technical lemma on primitive pairs that will be repeatedly used in the analysis of the birational geometry of anticanonical hypersurfaces.

\subsection{Smooth Fano polytopes and primitive collections}
Let $N \simeq \mathbb{Z}^n$ be a lattice with dual $M = \operatorname{Hom}(N, \mathbb{Z})$, and denote the associated real vector spaces by $N_{\mathbb{R}}$ and $M_{\mathbb{R}}$ respectively.
A \emph{lattice polytope} $\Delta \subseteq N_{\mathbb{R}}$ is the convex hull of a finite set of points in $N$. 
Whenever $0\in N$ is an interior point of $\Delta$, its 
\emph{polar polytope} is
\[
 \Delta^{\circ} = \{ u \in M_{\mathbb{R}} \mid \langle x,u \rangle \ge -1 \text{ for all } x \in \Delta \}.
\]
We say that $\Delta$ is \emph{reflexive} if the origin is an interior point and the polar $\Delta^{\circ}$ is itself a lattice polytope.
Refining this further, $\Delta$ is called \emph{terminal} if the only lattice points it contains are the origin and its vertices, that is, $\Delta \cap N = \{0\} \cup \operatorname{Vert}(\Delta)$.
Additionally, $\Delta$ is \emph{smooth} if it is simplicial, contains the origin in its interior, and the vertices of every facet form a basis of the lattice $N$.
These combinatorial definitions directly encode the geometry of the toric Fano variety $\mathbb P_{\Delta}$ associated with the face fan of a reflexive polytope $\Delta$. In particular, the variety is $\mathbb{Q}$-factorial (resp. smooth) if and only if $\Delta$ is simplicial (resp. smooth), and it has terminal singularities precisely when $\Delta$ is terminal. Furthermore, the Picard rank of $\mathbb P_{\Delta}$ is given by $r - n$, where $r = |\operatorname{Vert}(\Delta)|$.

\subsection{Primitive pairs and the irrelevant locus}
Let $\Delta\subseteq N_\mathbb R$ be a reflexive 
lattice polytope. After fixing an ordering $v_1,\dots,v_r$ of the vertices, we write $x_i:=x_{v_i}$.
The Cox ring of $\mathbb P_{\Delta}$ is the polynomial ring
\[
 \mathcal R(\mathbb P_{\Delta})
 =
 \mathbb{K}[x_1,\dots,x_r],
\]
and its irrelevant locus $\operatorname{Irr}(\mathbb P_{\Delta})$ is the union of linear subspaces $V(x_{v} \mid v \notin F)$ as $F$ varies among the facets of $\Delta$. Note that these subspaces have codimension at least $2$.
Let $L \subset \operatorname{Irr}(\mathbb P_{\Delta})$ be a codimension-$2$ irreducible component, defined by a prime ideal $I_{L}$.
We define the \emph{multiplicity of $L$} as
\[
\mu(L):=\max\{\,m\in\mathbb Z_{\ge 0}\mid \mathcal{R}(\mathbb P_{\Delta})_{[-K]}\subseteq I_{L}^{m}\,\}.
\]
Equivalently, $\mu(L)$ corresponds to the vanishing order of a general anticanonical section along $L$.
In the context of a smooth toric Fano variety $\mathbb P_\Delta$ one always has
\[
\mu(L)\in\{1,2\}.
\]
The upper bound $\mu(L)\le 2$ holds because the monomial $\prod_{v\in\operatorname{Vert}(\Delta)} x_v$ represents the anticanonical class and vanishes to order $2$ along any intersection of two coordinate hyperplanes. Conversely, since $-K$ is ample, every monomial of degree $[-K]$ vanishes along $L$ (i.e. lies in $I_L$), which implies the lower bound $\mu(L)\ge 1$.

The irreducible components of the irrelevant locus are in bijective correspondence with \emph{primitive collections},
see~\cite{bat}. A non-empty subset of vertices $\mathcal{P} \subseteq \operatorname{Vert}(\Delta)$ is defined as a primitive collection if it is not contained in any face of $\Delta$, whereas every proper subset of $\mathcal{P}$ is.
Each primitive collection $\mathcal{P} = \{v_{i_1}, \dots, v_{i_k}\}$ satisfies a unique linear dependency called the \emph{primitive relation}:
\[
    v_{i_1} + \dots + v_{i_k} \;=\; c_1 v_{j_1} + \dots + c_\ell v_{j_\ell},
\]
where the coefficients $c_r$ are positive integers and the sum on the right lies in the cone generated by $\{v_{j_1}, \dots, v_{j_\ell}\}$ in the fan of $\Delta$.
From this relation, we derive the \emph{degree} of the collection, defined as
\[
 d(\mathcal{P}) := k - \sum_{r=1}^{\ell} c_r.
\]
Geometrically, $\mathcal{P}$ determines the numerical class of a torus-invariant curve $C_{\mathcal{P}} \subset \mathbb{P}_{\Delta}$, and the integer $d(\mathcal{P})$ coincides with the anticanonical degree $-K_{\mathbb{P}_{\Delta}} \cdot C_{\mathcal{P}}$. Since $\mathbb{P}_{\Delta}$ is Fano, this degree must be strictly positive.
When considering collections of cardinality two—termed \emph{primitive pairs}—the condition $d(\mathcal{P}) > 0$ restricts the primitive relation to exactly two forms:
\begin{itemize}
    \item \textbf{Degree 2:} $v_i + v_j = 0$.
    \item \textbf{Degree 1:} $v_i + v_j = v_k$.
\end{itemize}

For a smooth Fano polytope $\Delta$, there is a bijection between primitive pairs and codimension-$2$ irreducible components of the irrelevant locus $\mathrm{Irr}(\mathbb{P}_{\Delta})$. Since the irrelevant locus is defined as the union of subspaces $V(x_v \mid v \in \mathcal{P})$ over all primitive collections, the components of codimension $2$ correspond precisely to collections of cardinality two. Thus, we identify a primitive pair $\{v_1, v_2\}$ with the linear subspace $L = V(x_1, x_2)$.

\begin{proposition}\label{prop:facet-irr-correspondence}
Let $\{v_1,v_2\}$ be a primitive pair of a smooth Fano polytope $\Delta$, and let
$L:=V(x_1,x_2)\subseteq \operatorname{Irr}(\mathbb P_\Delta)$. Then
\[
 d(\{v_1,v_2\})=\mu(L).
\]
\end{proposition}

\begin{proof}
We distinguish the two possible cases.

Assume first that $d(\{v_1,v_2\})=2$. Then the primitive relation is
$v_1+v_2=0$. Let $m\in \Delta^\circ\cap M$. The corresponding anticanonical
monomial is
\[
x^m=\prod_{v\in \operatorname{Vert}(\Delta)} x_v^{\langle m,v\rangle+1}.
\]
The exponents of $x_1$ and $x_2$ add up to
$(\langle m,v_1\rangle+1)+(\langle m,v_2\rangle+1)=
\langle m,v_1+v_2\rangle+2=2$.
Therefore every anticanonical monomial has total degree exactly $2$ in the
variables $x_1$ and $x_2$. In particular, each such monomial is divisible by
$x_1^2$, $x_1x_2$, or $x_2^2$, hence vanishes along $L$ with multiplicity $2$.
Thus $\mu(L)=2$.

Assume now that $d(\{v_1,v_2\})=1$. Up to relabelling, the primitive relation is
$v_1+v_2=v_3$. Choose $m\in \Delta^\circ\cap M$ such that
$\langle m,v_3\rangle=-1$, that is, let $m$ correspond to a facet of $\Delta$
containing $v_3$. By linearity we get
$\langle m,v_1\rangle+\langle m,v_2\rangle=\langle m,v_3\rangle=-1$.
Since $m\in \Delta^\circ$, one has $\langle m,v_i\rangle\ge -1$ for every vertex
$v_i$. Hence, up to swapping $v_1$ and $v_2$, the only possibility is
$\langle m,v_1\rangle=-1$ and $\langle m,v_2\rangle=0$. It follows that the
monomial $x^m$ has exponent $0$ in $x_1$ and exponent $1$ in $x_2$, so it
vanishes along $L$ with multiplicity $1$.

On the other hand, since $L$ is contained in the irrelevant locus, every
anticanonical monomial vanishes along $L$ at least once. Therefore
$\mu(L)=1$, and the proof is complete.
\end{proof}

\begin{remark}
The degree-$2$ case admits a convenient reformulation in terms of lattice width.
Let $v\in N$ be a primitive vector and let $P\subseteq M_\mathbb R$ be a
full-dimensional polytope. The width of $P$ along $v$ is defined as
$\operatorname{lw}_v(P):=\mathrm{length}(\varphi_v(P))$, where
$\varphi_v(x):=\langle x,v\rangle$. The \emph{lattice width} of $P$ is then
\[
 \operatorname{lw}(P):=\min\{\operatorname{lw}_v(P)\mid v\in N \text{ primitive},\, v\neq 0\}.
\]
If $\Delta$ is terminal and reflexive, then $\{v_1,-v_1\}$ is a primitive pair
if and only if $\operatorname{lw}_{v_1}(\Delta^\circ)=2$. Indeed, if
$\{v_1,-v_1\}$ is a primitive pair, then for every $x\in \Delta^\circ$ one has
$\langle x,v_1\rangle\ge -1$ and $\langle x,-v_1\rangle\ge -1$, hence
$\varphi_{v_1}(\Delta^\circ)\subseteq [-1,1]$. Since $\Delta^\circ$ has
non-empty interior, the width is exactly $2$.
Conversely, if $\operatorname{lw}_v(\Delta^\circ)=2$ for some primitive vector
$v\in N$, then $\varphi_v(\Delta^\circ)=[-1,1]$, so $\pm v\in(\Delta^\circ)^\circ=\Delta$.
Since $\Delta$ is terminal, every non-zero lattice point of $\Delta$ is a
vertex. Thus $\pm v\in \operatorname{Vert}(\Delta)$, and they form a primitive
pair.
\end{remark}

\subsection{A technical lemma on primitive pairs}
We conclude this section with a technical lemma that will be used repeatedly in Section~\ref{sec:3}.
\begin{lemma}\label{lem:unique-vertex-pair}
Let $\Delta$ be a smooth Fano lattice polytope. Let $v_1, v_2, v_3$ be vertices satisfying $v_1 + v_2 = v_3$, and assume that $v_3$ is not contained in any primitive pair.
Then, for any vertex $v_i$ ($i \ge 3$), up to swapping $v_1$ and $v_2$, there exists a dual lattice vector $m \in \Delta^{\circ} \cap M$ such that:
\[
    \langle m, v_1 \rangle = 0, \qquad \langle m, v_2 \rangle = -1, \qquad \langle m, v_i \rangle = -1.
\]
\end{lemma}

\begin{proof}
We distinguish two cases to locate the required facet normal $m \in \Delta^{\circ} \cap M$.
If $i=3$, we simply select any facet containing the vertex $v_3$.
If $i > 3$, we consider the set $\{v_3, v_i\}$. Since $v_3$ does not belong to any primitive pair, $\{v_3, v_i\}$ cannot be a primitive pair and must therefore be contained in a facet.
In either case, let $m \in \Delta^{\circ} \cap M$ be the defining normal of such a facet; by construction, we have $\langle m, v_3 \rangle = -1$ and $\langle m, v_i \rangle = -1$.
Finally, evaluating $m$ on the relation $v_1 + v_2 = v_3$ yields $\langle m, v_1 \rangle + \langle m, v_2 \rangle = -1$. Since $\langle m, \cdot \rangle \ge -1$ on $\Delta$, the only integer solution is $\{-1, 0\}$, which proves the claim.
\end{proof}

\section{Cox rings of embedded varieties}
\label{sec:2}

In this section we recall the embedded-variety method of~\cite{hlu} for computing Cox rings. 
Starting from a projective subvariety $X\subseteq Z$ of a Mori dream space, under the assumption that the pullback 
$\operatorname{Cl}(Z)\to \operatorname{Cl}(X)$ is an isomorphism, the Cox sheaf of $X$ can be identified with the pushforward of the structure sheaf of the inverse image of $X$ in the characteristic space of $Z$. 
This allows one to realize the characteristic space of $X$ inside the total coordinate space of $Z$ and to describe the Cox ring of $X$ as an intersection of localizations of a quotient of $\mathcal R(Z)$. 
We first recall the necessary language of Cox sheaves, characteristic spaces and irrelevant loci, and then state the precise form of the localization formula that will be used later. 
Finally, we record an algebraic generation criterion, derived from~\cite[Cor.~2.5]{hlu}, which gives a practical way to pass from the abstract intersection of localizations to explicit presentations by adjoining suitable fractions and checking a codimension condition.

\subsection{Characteristic spaces and Cox sheaves}
Let $\mathbb{K}$ be an algebraically closed field of characteristic $0$.  
Let $X$ be a normal variety such that $H^0(X,\mathcal{O}^*) = \mathbb{K}^*$.  
Denote by ${\rm WDiv}(X)$ the group of Weil divisors of $X$ and by ${\rm PDiv}(X)$ the subgroup of principal divisors.  
The {\em divisor class group} of $X$ is defined as  
\[
 {\rm Cl}(X) := \frac{{\rm WDiv}(X)}{{\rm PDiv}(X)}.
\]
Let $K\subseteq {\rm WDiv}(X)$ be a finitely
generated subgroup.
Assume that the quotient map $K \to {\rm Cl}(X)$ is an isomorphism. Define the {\em Cox sheaf} 
and the {\em Cox ring} of $X$ as:
\[
 \mathcal{R}_X := \bigoplus_{D \in K} \mathcal{O}_X(D)
 \hspace{2cm}
 \mathcal{R}(X) := H^0(X, \mathcal{R}_X).
\]
Being the variety $\mathbb{Q}$-factorial, the Cox sheaf is locally of finite type~\cite[\S 1.6.1]{adhl}. Denote by 
\[
 \widehat{X} := {\rm Spec}_X \mathcal{R}_X
 \hspace{2cm}
 \overline X := {\rm Spec}\,\mathcal R(X)
\]  
the {\em characteristic space} and the 
{\em total coordinate space} of $X$, respectively.
Both are irreducible and normal.
The variety $X$ can be covered by open subsets of the form  
$X_{D, f} := X \setminus {\rm Supp}({\rm div}(\tilde{f}) + D)$, where $D\in K$. The natural map  
\[
 H^0(X, \mathcal{R}_X)_f \to H^0(X_{D, f}, \mathcal{R}_X)
\]  
is an isomorphism of graded algebras. Gluing 
we obtain an embedding $\widehat{X} \to \overline{X}$, whose complement is a subset of codimension at least two, called the  
{\em irrelevant locus} of 
$X$~\cite[Constr.~1.6.3.1]{adhl}.
The inclusion $\mathcal{O}_X \to \mathcal{R}_X$, as the degree-zero part of the Cox sheaf, induces a morphism $p_X\colon \widehat{X} \to X$, which is a good quotient for the action of the quasitorus $H_X := {\rm Spec}\, \mathbb{K}[{\rm Cl}(X)]$.  
This situation is summarized in the following diagram:  
\[
\begin{tikzcd}
\widehat{X} \arrow[r, hookrightarrow] \arrow[d, "p_X"'] & \overline{X} \\
X
\end{tikzcd}
\]  

\subsection{The embedded-variety method}
Let $Z$ be a smooth projective variety with a finitely generated Cox ring, and let $p_Z \colon \widehat{Z} \to Z$ be its characteristic space.
Given a projective subvariety $X \subseteq Z$, 
let
\[
\widehat{X}_Z := p_Z^{-1}(X)
\]
and let $p_X\colon \widehat{X}_Z \to X$
denote the restriction of $p_Z$. 
If $X$ is normal and the pullback 
$\imath^* \colon \operatorname{Cl}(Z) \to \operatorname{Cl}(X)$
is an isomorphism, then there is an isomorphism of ${\rm Cl}(X)$-graded 
$\mathcal{O}_X$-algebras
\[
 \mathcal{R}_X \simeq (p_X)_* \mathcal{O}_{\widehat{X}_Z}.
\]
In particular, $\widehat{X}_Z$ identifies with the characteristic space of $X$. 
By abuse of notation, we shall therefore write $\widehat{X}$ in place of $\widehat{X}_Z$ in the sequel.
Let $\tilde{X} \subseteq \overline{Z}$ be the closure of $\widehat{X}$. We define $\tilde{X}_1 \subseteq \tilde{X}$ as the open subset obtained by removing the union of all codimension-one subvarieties of $\tilde{X} \setminus \widehat{X}$. The following diagram represents this setup:
\[
\begin{tikzcd}
\widehat{X} \arrow[r, hook] \arrow[d, hook] 
& \tilde{X}_1 \arrow[r, hook] 
& \tilde{X} \arrow[d, hook] \\
\widehat{Z} \arrow[rr, hook] && 
\overline{Z}
\end{tikzcd}
\]
In this diagram, the complement $\tilde{X}_1 \setminus \widehat{X}$ has codimension at least two, while the complement $\tilde{X} \setminus \tilde{X}_1$ is a union of hypersurfaces.

\begin{theorem}
\label{thm-emb}
Under the above hypotheses let $\{m_1, \dots, m_k\}\subseteq\mathcal R(Z)$ 
be a homogeneous basis for the defining ideal 
of $\tilde X\setminus\tilde X_1$. 
If $\tilde{X}_1$ is normal, then 
\[
 \mathcal{R}(X)
 =
 \bigcap_{i=1}^k (\mathcal{R}(Z)/I(\tilde{X}))_{m_i}.
\]
\end{theorem}

In Theorem~\ref{thm-emb}, if the coordinate ring of $\tilde{X} \subseteq \overline{Z}$ is Cohen-Macaulay, and in particular if it is a complete intersection, then Serre's condition $S_2$ is automatically satisfied~\cite[\href{https://stacks.math.columbia.edu/tag/0342}{Tag 0342}]{stacks-project}. Since this condition is local~\cite[\href{https://stacks.math.columbia.edu/tag/033Q}{Tag 033Q}]{stacks-project}, we conclude that in this case, $\tilde{X}_1$ is normal if and only if it is smooth in codimension one.


We conclude the section with an algebraic criterion that will be used in the
next section to turn the description by localizations into explicit
presentations. The result 
follows directly from \cite[Cor.~2.5]{hlu}.

\begin{proposition}
\label{pro:gens}
Let $R$ be a normal finitely generated $\mathbb K$-algebra, and let
$m_1,\dots,m_k\in R$ be non-zero elements. Assume that
$q_j:=f_j/m_1^{r_j}$ belongs to $R_{m_1}\cap\cdots\cap R_{m_k}$ for
$j=1,\dots,t$. Let $R'\subseteq R_{m_1}\cap\cdots\cap R_{m_k}$ be the
subring generated by $R$ and the elements $q_1,\dots,q_t$.
Set $P:=R[s_1,\dots,s_t]$ and
\[
I:=\langle s_1m_1^{r_1}-f_1,\dots,s_tm_1^{r_t}-f_t\rangle.
\]
Then
\[
R'\cong P/(I:\langle m_1\rangle^\infty).
\]
Assume moreover that
\[
\operatorname{codim}_{\operatorname{Spec}(R')}V(m_1,\dots,m_k)\geq 2.
\]
Then the normalization of $R'$ in $\operatorname{Frac}(R)$ is
\[
R'^\nu=R_{m_1}\cap\cdots\cap R_{m_k}.
\]
\end{proposition}

\section{Proof of Theorem~\ref{thm:1}}
\label{sec:3}

This section is devoted to the proof of Theorem~\ref{thm:1}. 
The common strategy is to apply the embedded-variety method recalled in Section~\ref{sec:2} to a general anticanonical hypersurface 
\(X\subseteq \mathbb P_\Delta\). 
After identifying \(\mathcal R(X)\) with an intersection of localizations of 
\(R=\mathbb K[x_1,\dots,x_r]/\langle f\rangle\), we introduce suitable fractions which generate this intersection in each of the combinatorial cases appearing in the statement. 
The proof then reduces to showing that the resulting presentation is already saturated with respect to the relevant monomial and defines a normal ring. 
For the complete-intersection cases this is achieved by a general normality and saturation criterion, Lemma~\ref{lem:ci-normal-sat}; in the Pfaffian case we use instead Lemma~\ref{lem:pfaffian-normal}, based on the Buchsbaum--Eisenbud theorem. 
The required codimension estimates are obtained from the primitive relations of the polytope and from the technical lemma on primitive pairs proved in Section~\ref{sec:1}. 
We conclude the section by discussing a variant, corresponding to the index~\(117\), where the same argument does not apply verbatim and the Cox ring has to be described by a different set of generators and relations.

\subsection{General strategy}

\begin{proof}[Proof of Thm.~\ref{thm:1}]
We proceed by applying Proposition~\ref{pro:gens} to the ring \(R := \mathbb{K}[x_1,\dots,x_r]/\langle f\rangle\). Let \(\mathcal{M} = \{m_1,\dots,m_k\}\) denote the minimal squarefree monomial generators of the codimension-\(2\) component of the irrelevant ideal.
Identifying the Cox ring \(\mathcal{R}(X)\) with the intersection of localizations \(R_{m_1} \cap \dots \cap R_{m_k}\), we introduce auxiliary variables \(s_1, \dots, s_t\) corresponding to generators \(g_1/m_1^{r_1},\dots,g_t/m_1^{r_t}\) of this intersection.
We then consider the ideal \(I \subseteq \mathbb{K}[x,s]\) generated by the elements \(s_j m_1^{r_j}-g_j\).
Provided that the locus \(V(m_1,\dots,m_k)\) has codimension at least \(2\) inside \(\operatorname{Spec}(\mathbb{K}[x,s]/I)\), the proposition yields the isomorphism
\[
    \mathcal{R}(X) \;\simeq\; \frac{\mathbb{K}[x_1,\dots,x_r,s_1,\dots,s_t]}{I : \langle m_1\rangle^\infty}.
\]
We will explicitly show that the saturation is redundant, meaning that \(I:\langle m_1\rangle^\infty = I\).

\subsection{Cases \texorpdfstring{$(i)$}{(i)} and \texorpdfstring{$(ii)$}{(ii)}}

\textbf{Case (i).}
We assume that $\{v_1,v_2\}$ is the only primitive pair and that
$v_1+v_2=v_3$ is the corresponding primitive relation. In particular,
$V(x_1,x_2)$ is the unique component of codimension~$2$ in the irrelevant
locus of $\mathbb P_\Delta$.
By Proposition~\ref{prop:facet-irr-correspondence}, the multiplicity of a
general polynomial $f$ of degree $-K$ along $V(x_1,x_2)$ is~$1$. Hence we can
write
\[
    f=f_1x_1+f_2x_2,
\]
where $f_2\in \mathbb K[x_2,\dots,x_r]$. We consider the ideal
\[
I=\langle f_2-sx_1,\ f_1+sx_2\rangle
\subseteq \mathbb K[x_1,\dots,x_r,s].
\]
We first prove the required codimension estimate. It is enough to show that
\[
    \operatorname{codim} V(x_1,x_2,f_1,f_2)=4.
\]
After setting $x_1=x_2=0$, this amounts to showing that the restrictions of
$f_1$ and $f_2$ have no common variable factor. Equivalently, for every
$i>2$, at least one of $f_1$ and $f_2$ does not belong to
$\langle x_1,x_2,x_i\rangle$.
This follows from Lemma~\ref{lem:unique-vertex-pair}. Indeed, for every
$i\geq 3$ there exists a lattice point $m\in \Delta^\circ\cap M$ such that
\[
\langle m,v_1\rangle=0,
\qquad
\langle m,v_2\rangle=\langle m,v_i\rangle=-1.
\]
The corresponding anticanonical monomial
$\prod_j x_j^{\langle m,v_j\rangle+1}$ is divisible by $x_1$, but not by
$x_2$ or $x_i$. Therefore, after factoring out $x_1$, it gives a monomial of
$f_1$ which survives modulo $\langle x_1,x_2\rangle$ and is not divisible by
$x_i$. Hence $f_1\notin \langle x_1,x_2,x_i\rangle$ for every $i>2$, and the
codimension estimate follows.
We now apply Lemma~\ref{lem:ci-normal-sat}{(i)}. The
estimate just proved gives
$\operatorname{codim} V(I,x_1,x_2)=4$, while the corresponding estimates for
$V(I,x_1)$ and $V(I,x_2)$ follow from the same generality argument. On the
open set $D(x_1x_2)$, the variable $s$ can be eliminated, and the quotient
identifies with a localization of
$\mathbb K[x_1,\dots,x_r]/\langle f\rangle$. Since $f$ is general, this
localization is a domain and is regular in codimension one. Thus all
hypotheses of Lemma~\ref{lem:ci-normal-sat} are satisfied. Consequently,
\[
I:\langle x_1x_2\rangle^\infty=I,
\]
and the quotient displayed in Case~(i) is normal.

\medskip

\textbf{Case (ii).}
We assume that $\{v_1,v_2\}$ is the only primitive pair and 
$v_1 + v_2 = 0$ is the corresponding primitive relation.
In particular, $V(x_1,x_2)$ is the unique component of codimension~$2$
in the irrelevant locus of $\mathbb{P}_\Delta$.
By Proposition~\ref{prop:facet-irr-correspondence}, the multiplicity of a general
polynomial $f$ of degree $-K$ along $V(x_1,x_2)$ is~$2$, so we can write
\[
    f \;=\; f_1\,x_1^2 \;+\; f_2\,x_1 x_2 \;+\; f_3\,x_2^2,
\]
and on \(V(x_1,x_2) \subseteq \mathbb{K}^r\) the equations
\(f_1 = f_2 = f_3 = 0\) define three hypersurfaces in the coordinates
\(x_3,\dots,x_r\).
By the generality assumption, either \(V(f_1,f_2,f_3)\) has codimension~$3$,
or \(f_1,f_2,f_3 \in \langle x_i,x_j\rangle\) for some \(i,j>2\).
The latter cannot occur because it would produce a second primitive pair,
contradicting the uniqueness assumption. Therefore
\[
    \mathrm{codim}\,V(x_1,x_2,f_1,f_2,f_3) \;=\; 5.
\]

Now let
\[
I=\langle f_3-x_1s_1,\ f_2+x_1s_2+x_2s_1,\ f_1-x_2s_2\rangle
\subseteq \mathbb K[x_1,\dots,x_r,s_1,s_2].
\]
The last codimension computation gives
\[
\operatorname{codim}V(I,x_1,x_2)=5.
\]
The corresponding estimates for \(V(I,x_1)\) and \(V(I,x_2)\) are the expected
ones and follow from the same generality argument. On \(D(x_1x_2)\), the
relations eliminate \(s_1\) and \(s_2\), and the quotient identifies with a
localization of \(\mathbb K[x_1,\dots,x_r]/\langle f\rangle\). Since \(f\) is
general, this localization is a domain and is regular in codimension one.
Thus the hypotheses of Lemma~\ref{lem:ci-normal-sat}{(ii)} are satisfied. We conclude
that
\[
I:\langle x_1x_2\rangle^\infty=I,
\]
and that the quotient displayed in Case~(ii) is normal.

\subsection{Cases \texorpdfstring{$(iii)$}{(iii)} and \texorpdfstring{$(iv)$}{(iv)}}
\textbf{Case (iii).}
From the relation $v_1+v_2=0$ we deduce that $f$ can be written in the form
\[
    f \;=\; g_1\,x_1^2 \;+\; g_2\,x_1x_2 \;+\; g_3\,x_2^2,
\]
where $g_1,g_2,g_3\in \mathbb K[x_3,\dots,x_r]$. We now use the remaining
primitive relations to refine this expression. Any monomial appearing in
$g_1x_1^2$ satisfies $\langle m,v_1\rangle=1$, hence
$\langle m,v_3\rangle=\langle m,v_4\rangle+1\geq 0$. Thus $x_3$ divides
$g_1$. Analogously, $x_4$ divides $g_3$. Therefore \(f\) has the form
\[
    f \;=\; f_1\,x_1^{2}x_3 \;+\; f_2\,x_1x_2 \;+\; f_3\,x_2^{2}x_4.
\]
In this case the codimension-$2$ part of the irrelevant ideal is defined by
$\langle x_1x_2,\ x_1x_3,\ x_2x_4\rangle$.
Using the relation \(f=0\), one checks that \(f_3/x_1\) belongs to the
localizations corresponding to the components \(V(x_1,x_3)\) and
\(V(x_2,x_4)\). Similarly, \(f_1/x_2\) belongs to the same intersection of
localizations. We denote these two elements by \(s_1\) and \(s_2\), respectively.
Let
\[
I=\langle
f_3-x_1s_1,\ 
f_2+x_1x_3s_2+x_2x_4s_1,\ 
f_1-x_2s_2
\rangle
\subseteq \mathbb K[x_1,\dots,x_r,s_1,s_2].
\]
We now verify the codimension condition along the codimension-$2$ part of the
irrelevant locus. Intersecting with the three components
$V(x_1,x_2)$, $V(x_1,x_3)$ and $V(x_2,x_4)$, respectively, gives
\[
\begin{array}{c}
V(x_1,x_2,f_1,f_2,f_3),\\[3pt]
V(x_1,x_3,f_3,x_2x_4s_1+f_2,x_2s_2-f_1),\\[3pt]
V(x_2,x_4,x_1s_1-f_3,x_1x_3s_2+f_2,f_1).
\end{array}
\]
We claim that each of these has dimension \(r-3\) in
\(\operatorname{Spec}\mathbb K[x_1,\dots,x_r,s_1,s_2]\). For the last two
sets this follows from the generality of \(f_1,f_2,f_3\) and from the presence
of the variables \(s_1,s_2\), which make the displayed equations impose the
expected number of conditions.
It remains to discuss the first set. Since \(f_1,f_2,f_3\) are general
sections of their respective monomial linear systems, a failure of the expected
codimension can only occur along the base locus of these systems. Since the
linear systems are generated by monomials, their base loci are coordinate
subspaces. Thus \(V(x_1,x_2,f_1,f_2,f_3)\) could have codimension \(4\) only if
there were variables \(x_i,x_j\) such that
\(f_1,f_2,f_3\in \langle x_i,x_j\rangle\). This would force
\(\{v_i,v_j\}\) to be a primitive pair, contradicting the list of primitive
pairs in the present case. Hence the first set also has the expected dimension
\(r-3\).

Finally, we apply Lemma~\ref{lem:ci-normal-sat}{(iii)}. The codimension estimate above gives the condition on
\(V(I,x_1,x_2)\), while the corresponding estimates for \(V(I,x_1)\) and
\(V(I,x_2)\) follow from the same generality argument. On the open set
\(D(x_1x_2)\), the variables \(s_1\) and \(s_2\) can be eliminated, and the
quotient identifies with a localization of
\(\mathbb K[x_1,\dots,x_r]/\langle f\rangle\). Since \(f\) is general, this
localization is a domain; moreover, the same local description and the
generality assumption give regularity in codimension one on \(D(x_1)\cup
D(x_2)\). Therefore all hypotheses of Lemma~\ref{lem:ci-normal-sat} are
satisfied. We conclude that
\[
I:\langle x_1x_2\rangle^\infty=I,
\]
and that the quotient displayed in Case~(iii) is normal.

\textbf{Case (iv).}
Since $\{v_1,v_2\}$ is a primitive pair,
every monomial of $f$ is divisible by either $x_1$ or $x_2$,
and the same holds for $x_3$ and $x_4$.
Therefore we can write
\[
    f \;=\;
    f_1\,x_1x_3 \;+\;
    f_2\,x_1x_4 \;+\;
    f_3\,x_2x_3 \;+\;
    f_4\,x_2x_4,
\]
where each $f_i$ is homogeneous of the appropriate degree in
$\mathbb K[x_1,\dots,x_r]$.
The codimension-$2$ component of the irrelevant ideal is
\(\langle x_1x_3,\, x_1x_4,\, x_2x_3,\, x_2x_4\rangle\).
From the above expression for $f$ we obtain that
\[
    \frac{x_1f_2 + x_2f_4}{x_3}
    \quad\text{and}\quad
    \frac{x_3f_3 + x_4f_4}{x_1}
\]
belong to $\mathcal{R}(X)$.
If we denote these elements by $s_1$ and $s_2$, respectively,
we recover the relations appearing in the statement.
Intersecting with the coordinate subspaces $V(x_1,x_2)$ and $V(x_3,x_4)$
we obtain three components
\[
\begin{array}{c}
    V(x_1,\ x_2,\ x_3f_1 + x_4f_2,\ 
    x_3f_3 + x_4f_4,\ 
    f_1f_4 - f_2f_3,\ 
    s_1),\\[4pt]
    V(x_3,\ x_4,\ x_1f_1 + x_2f_3,\ 
    x_1f_2 + x_2f_4,\ 
    f_1f_4 - f_2f_3,\ 
    s_2),\\[4pt]
    V(x_1,\ x_2,\ x_3,\ x_4,\ f_1f_4 - f_2f_3 + s_1s_2).
\end{array}
\]
To conclude we must show that each of these
components has dimension $r-3$.
For the third component this is clear, since the last equation involves only one monomial in the variables $s_1$ and $s_2$, and hence it cannot be dependent on the previous four.
Concerning the first two components, they both have dimension $r-3$
if and only if the corresponding varieties
\[
\begin{array}{l}
    V(x_1,\ x_2,\ x_3f_1 + x_4f_2,\
      x_3f_3 + x_4f_4,\
      f_1f_4 - f_2f_3),\\[4pt]
    V(x_3,\ x_4,\ x_1f_1 + x_2f_3,\
      x_1f_2 + x_2f_4,\
      f_1f_4 - f_2f_3)
\end{array}
\]
have codimension~$4$.
We focus on the first one, since the argument for the second is analogous.

Assume, by contradiction, that the first component has codimension~$3$.
Then, after evaluating at $x_1 = x_2 = 0$, the remaining polynomials
would have a nontrivial common factor.
After a generic choice of coefficients we may assume that \(A:=x_3f_1+x_4f_2\) and \(B:=x_3f_3+x_4f_4\) have disjoint monomial supports; varying one coefficient of \(B\) (so \(B=B_t\)) shows that a non-monomial gcd would force a fixed non-monomial irreducible \(P\) to divide both \(A\) and \(B_t\) for infinitely many \(t\), but \(P\mid B_t\) is a proper Zariski-closed condition and cannot hold generically. Hence the gcd is monomial, so some variable \(x_i\) divides both \(A\) and \(B\).
If $i = 3$ (the case $i = 4$ being similar), then
$f_2, f_4 \in \langle x_1, x_2, x_3\rangle$;
while if $i > 4$ we would have
$f_1, f_2, f_3, f_4 \in \langle x_1, x_2, x_i\rangle$.
We claim that both possibilities are excluded by Lemma~\ref{lem:unique-vertex-pair},
whose hypotheses are satisfied in our setting.
Recall that lattice points $m\in \Delta^\circ\cap M$ correspond to Laurent monomials
\[
x^m \;=\; \prod_{j=1}^r x_j^{\langle m,v_j\rangle+1}
\in H^0(\mathbb{P}_\Delta,-K),
\]
and, in the decomposition
\[
f \;=\; f_1\,x_1x_3 \;+\; f_2\,x_1x_4 \;+\; f_3\,x_2x_3 \;+\; f_4\,x_2x_4,
\]
a monomial $x^m$ contributes to $f_1$ (resp. $f_2,f_3,f_4$) iff
$\langle m,v_1\rangle,\langle m,v_3\rangle \ge 0$
(resp. $\langle m,v_1\rangle,\langle m,v_4\rangle \ge 0$;
$\langle m,v_2\rangle,\langle m,v_3\rangle \ge 0$;
$\langle m,v_2\rangle,\langle m,v_4\rangle \ge 0$).
Moreover, since $\{v_1,v_2\}$ is primitive, every monomial has
$\max\{\langle m,v_1\rangle,\langle m,v_2\rangle\}\ge 0$; and by the analogous
property for $x_3,x_4$, every monomial also satisfies
$\max\{\langle m,v_3\rangle,\langle m,v_4\rangle\}\ge 0$.

\smallskip
\emph{Case $i=3$.}
Apply Lemma~\ref{lem:unique-vertex-pair} with any vertex $v_{i'}\notin\{v_1,v_2,v_3\}$.
We obtain $m\in \Delta^\circ\cap M$ such that
\[
\langle m,v_1\rangle=0, \qquad
\langle m,v_2\rangle=-1, \qquad
\text{and hence } \langle m,v_3\rangle=\langle m,v_1\rangle+\langle m,v_2\rangle=-1.
\]
This forces $\langle m,v_4\rangle\ge 0$.
Therefore $x^m$ contributes to $f_2$ (since $\langle m,v_1\rangle,\langle m,v_4\rangle\ge 0$),
and its coefficient in $f_2$ is not divisible by any of $x_1,x_2,x_3$
(because $\langle m,v_1\rangle=0$ and $\langle m,v_2\rangle=\langle m,v_3\rangle=-1$).
Hence $f_2\notin \langle x_1,x_2,x_3\rangle$.
Swapping $v_1$ and $v_2$ gives, by the lemma, a point $m'$ with
$\langle m',v_2\rangle=0$, $\langle m',v_1\rangle=\langle m',v_3\rangle=-1$ and
$\langle m',v_4\rangle\ge 0$, producing a monomial in $f_4$ not divisible by
$x_1,x_2,x_3$. Thus $f_4\notin \langle x_1,x_2,x_3\rangle$.
This contradicts the assumption $f_2,f_4\in \langle x_1,x_2,x_3\rangle$.

\smallskip
\emph{Case $i>4$.}
If we apply Lemma~\ref{lem:unique-vertex-pair} to $v_i$, we get 
$m\in \Delta^\circ\cap M$ such that
\[
\langle m,v_1\rangle=0,\qquad
\langle m,v_2\rangle=\langle m,v_i\rangle=-1,\qquad
\text{and hence }\langle m,v_3\rangle=-1.
\]
Again this implies $\langle m,v_4\rangle\ge 0$,
so $x^m$ contributes to $f_2$ and its coefficient is not divisible by
$x_1,x_2,x_i$ (since $\langle m,v_1\rangle=0$ and
$\langle m,v_2\rangle=\langle m,v_i\rangle=-1$).
Therefore $f_2\notin \langle x_1,x_2,x_i\rangle$.
This contradicts the assumption $f_1,f_2,f_3,f_4\in \langle x_1,x_2,x_i\rangle$.
In both cases we reach a contradiction, so the two possibilities are excluded,
as claimed.

We now apply Lemma~\ref{lem:pfaffian-normal} with \(a=1\) and \(b=3\). Indeed, the
Pfaffian ideal appearing in that lemma is precisely the ideal displayed in
the statement of the theorem in Case~(iv). Namely, if
\(S:=\mathbb K[x_1,\dots,x_r,s_1,s_2]\), this ideal is
\[
J=\left\langle
\begin{gathered}
f_1f_4-f_2f_3+s_1s_2,\quad
x_2s_2+x_3f_1+x_4f_2,\quad
x_1s_2-x_3f_3-x_4f_4,\\
x_1f_2+x_2f_4-x_3s_1,\quad
x_1f_1+x_2f_3+x_4s_1
\end{gathered}
\right\rangle.
\]
By the preceding codimension computation, the closed subset of
\(\operatorname{Spec}(S/J)\) lying over the codimension-two part of the
irrelevant locus has codimension at least \(2\). Moreover, on the open set
\(D(x_1x_3)\), the relations eliminate \(s_1\) and \(s_2\), and the quotient
identifies with
\[
\left(\mathbb K[x_1,\dots,x_r]/\langle f\rangle\right)_{x_1x_3}.
\]
Thus this localization is a domain and is regular in codimension one, because
\(f\) is general. Finally, for a general choice of the coefficients, the
Pfaffian ideal \(J\) has the expected codimension. Hence all hypotheses of
Lemma~\ref{lem:pfaffian-normal} are satisfied. We conclude that \(S/J\) is
normal and that
\[
J:\langle x_1x_3\rangle^\infty=J,
\]
and that the quotient displayed in Case (iv) is normal.
\end{proof}

\begin{lemma}
\label{lem:ci-normal-sat}
Let \(S\) and \(I\subseteq S\), $c\in\{2,3\}$ be one of the following:
\[
\begin{array}{ll}
\textup{(i)} &
S=\mathbb K[x_1,\dots,x_r,s],\quad
I=\langle f_2-x_1s,\ f_1+x_2s\rangle,\quad c=2;
\\[4pt]
\textup{(ii)} &
S=\mathbb K[x_1,\dots,x_r,s_1,s_2],\quad
I=\langle f_3-x_1s_1,\ f_2+x_1s_2+x_2s_1,\ f_1-x_2s_2\rangle,\quad c=3;
\\[4pt]
\textup{(iii)} &
S=\mathbb K[x_1,\dots,x_r,s_1,s_2],\quad
I=\langle f_3-x_1s_1,\ f_2+x_1x_3s_2+x_2x_4s_1,\ f_1-x_2s_2\rangle,\quad c=3.
\end{array}
\]
Assume that:
\begin{enumerate}
\item \(\operatorname{codim}_S V(I)=c\);
\item \(\operatorname{codim}_S V(I,x_1)\ge c+1\) and
\(\operatorname{codim}_S V(I,x_2)\ge c+1\);
\item \(\operatorname{codim}_S V(I,x_1,x_2)\ge c+2\);
\item \((S/I)_{x_1x_2}\) is a domain;
\item \(S/I\) is regular in codimension one on \(D(x_1)\cup D(x_2)\).
\end{enumerate}
Then \(S/I\) is a normal domain and
\[
I:\langle x_1x_2\rangle^\infty=I.
\]
\end{lemma}

\begin{proof}
Since \(S\) is a polynomial ring and \(I\) is generated by \(c\) elements of
codimension \(c\), the generators of \(I\) form a regular sequence. Thus
\(S/I\) is a complete intersection. In particular, it is Cohen--Macaulay and
satisfies Serre's condition \(S_2\).
We first prove that \(x_1x_2\) is a non-zero divisor on \(S/I\). Since \(S/I\) is
Cohen--Macaulay, all associated primes are minimal. The inequalities
\(\operatorname{codim}_S V(I,x_1)\ge c+1\) and
\(\operatorname{codim}_S V(I,x_2)\ge c+1\) show that neither \(x_1\) nor \(x_2\) is
contained in a minimal prime of \(I\). Hence both \(x_1\) and \(x_2\), and therefore
also \(x_1x_2\), are non-zero divisors on \(S/I\).
It follows that the natural map \(S/I\to (S/I)_{x_1x_2}\) is injective. Since
\((S/I)_{x_1x_2}\) is a domain by assumption, \(S/I\) is a domain as well. In
particular \(I\) is prime. Since \(x_1x_2\notin I\), primeness immediately gives
\[
 I:\langle x_1x_2\rangle^\infty=I.
\]
It remains to prove normality. By assumption
\(\operatorname{codim}_S V(I,x_1,x_2)\ge c+2\). Since \(I\) has codimension \(c\),
the closed subset defined by \(x_1=x_2=0\) has codimension at least \(2\) in
\(\operatorname{Spec}(S/I)\). Hence every codimension-one point of
\(\operatorname{Spec}(S/I)\) lies in \(D(x_1)\cup D(x_2)\). By assumption, the
ring \(S/I\) is regular in codimension one on this open subset. Therefore
\(S/I\) satisfies Serre's condition \(R_1\).
Thus \(S/I\) satisfies both \(S_2\) and \(R_1\). Since it is a domain, Serre's
criterion shows that \(S/I\) is normal.
\end{proof}

\begin{lemma}
\label{lem:pfaffian-normal}
Let $S:=\mathbb K[x_1,\dots,x_r,s_1,s_2]$,
and let \(f_1,\dots,f_4\in \mathbb K[x_1,\dots,x_r]\). Consider the
skew-symmetric matrix
\[
M=
\begin{pmatrix}
0   & x_1 & x_2 & x_3 & x_4\\
-x_1& 0   & -s_1& -f_4& f_3\\
-x_2& s_1 & 0   & f_2 & -f_1\\
-x_3& f_4 & -f_2& 0   & s_2\\
-x_4& -f_3& f_1 & -s_2& 0
\end{pmatrix},
\]
and let \(J\subseteq S\) be the ideal generated by the maximal Pfaffians of
\(M\). Let \(a\in\{1,2\}\) and \(b\in\{3,4\}\), and set
$X:=V(J)\subseteq\mathbb K^{r+2}$.
Assume that:
\begin{enumerate}
\item \(\operatorname{codim} X=3\);

\item \(\operatorname{codim}_{X}\bigl(X\cap V(x_a)\bigr)\geq 1\) and
\(\operatorname{codim}_{X}\bigl(X\cap V(x_b)\bigr)\geq 1\);

\item \(\operatorname{codim}_{X}\bigl(X\cap V(x_a,x_b)\bigr)\geq 2\);

\item \((S/J)_{x_a x_b}\) is a domain;

\item the open subset $X\setminus V(x_a,x_b)$
is regular in codimension one.
\end{enumerate}
Then \(S/J\) is normal, and
\[
J:\langle x_a x_b\rangle^\infty=J.
\]
\end{lemma}

\begin{proof}
Since \(J\) is the Pfaffian ideal of a \(5\times 5\) skew-symmetric matrix
and has the expected codimension \(3\), the Buchsbaum--Eisenbud theorem~\cite{be}
implies that \(S/J\) is Gorenstein. In particular, \(S/J\) is
Cohen--Macaulay. Hence \(S/J\) satisfies Serre's condition \(S_2\), and its
associated primes are precisely its minimal primes.
By assumption (2), neither \(X\cap V(x_a)\) nor \(X\cap V(x_b)\) contains an
irreducible component of \(X\). Equivalently, neither \(x_a\) nor \(x_b\) is
contained in a minimal prime of \(J\). Since \(S/J\) is Cohen--Macaulay, this
implies that \(x_a\) and \(x_b\) are non-zero divisors on \(S/J\). Therefore
\(x_a x_b\) is also a non-zero divisor on \(S/J\).
It follows that the localization map
\[
S/J \longrightarrow (S/J)_{x_a x_b}
\]
is injective. By assumption (4), \((S/J)_{x_a x_b}\) is a domain. Since
\(S/J\) injects into this domain, \(S/J\) itself is a domain.
Moreover, since \(x_a x_b\) is a non-zero divisor on \(S/J\), for every
\(n\geq 1\) we have $J:(x_a x_b)^n=J$. Thus
\[
J:\langle x_a x_b\rangle^\infty=J.
\]
It remains to prove normality. By assumption (3), the closed subset
$X\cap V(x_a,x_b)$
has codimension at least \(2\) in \(X\). Hence every codimension-one point of
\(X\) belongs to the open subset
$X\setminus V(x_a,x_b)$.
By assumption (5), this open subset is regular in codimension one. Therefore
\(S/J\) is regular at every codimension-one point of \(X\), that is, \(S/J\)
satisfies Serre's condition \(R_1\).
We have already shown that \(S/J\) satisfies \(S_2\). Since \(S/J\) is a
domain, Serre's criterion implies that \(S/J\) is normal.
\end{proof}

\subsection{A variant: the case of index \texorpdfstring{$117$}{117}}
\begin{remark}
    \label{rem:117}
The proof of Theorem~\ref{thm:1}, part (iv) works verbatim if the two relations are $v_1+v_2=v_3+v_4
= v_5$. On the other hand, 
if $v_1+v_2=v_3$ or $v_4$, we can not 
apply Lemma~\ref{lem:unique-vertex-pair},
so that the proof does not hold.
We are going to see in the following 
example that the Cox ring is indeed 
different (even if it is still finitely
generated).
\end{remark}

\begin{example}\label{ex:117}
Let $\Delta$ be the $4$-dimensional smooth lattice polytope with index $117$
in the Graded Ring Database, whose vertices are the columns of
\[
\begin{pmatrix*}[r]
 1 & -1 & 0 & 0 & 0 & 0 & 0\\
 0 & 0 & 0 & 0 & 0 & 1 &-1\\
 0 & 0 & 0 & 1 & 1 & 0 & -1\\
 0 & 1 & 1 & -1 & 0 & 0 & 0
\end{pmatrix*}.
\]
We have the primitive relations
\[
v_1+v_2=v_3,
\qquad
v_3+v_4=v_5.
\]
We first observe that these relations force certain divisibilities among the
coefficients of a general anticanonical equation. Indeed, any monomial
appearing in the term $f_1x_1x_3$ satisfies
$\langle m,v_1\rangle\geq 0$ and $\langle m,v_3\rangle\geq 0$. From the
relation $v_1+v_2=v_3$ it follows that either
$\langle m,v_2\rangle\geq 0$ or $\langle m,v_1\rangle\geq 1$. Hence every such
monomial is divisible by either $x_1$ or $x_2$, and therefore
$f_1\in \langle x_1,x_2\rangle$. The same argument applies to $f_3$.
Thus a general anticanonical hypersurface
$X\subseteq \mathbb P_\Delta$ can be written in the form
\[
f =
x_1^2x_3f_1
+x_1x_2x_3f_2
+x_1x_4f_3
+x_2^2x_3f_4
+x_2x_4f_5.
\]
The following elements belong to the intersection of the localizations
describing $\mathcal R(X)$:
\[
\frac{x_2x_3f_4+x_4f_5}{x_1},
\qquad
\frac{x_1f_3+x_2f_5}{x_3},
\qquad
-\frac{x_3f_1f_5+f_3s_1}{x_2}.
\]
We denote them by $s_1,s_2,s_3$, respectively.
Let $S:=\mathbb K[x_1,\dots,x_r,s_1,s_2,s_3]$,
and let 
\[
\begin{aligned}
I:=\big\langle\,
&x_1s_1-x_2x_3f_4-x_4f_5,\\
&x_1^2f_1+x_1x_2f_2+x_2^2f_4+x_4s_2,\\
&x_1f_1f_5+x_2f_2f_5-x_2f_3f_4+s_1s_2,\\
&x_1f_3+x_2f_5-x_3s_2,\\
&x_2s_3+x_3f_1f_5+f_3s_1,\\
&x_3^2f_1f_4+x_3f_2s_1-x_4s_3+s_1^2,\\
&x_1x_3f_1+x_2x_3f_2+x_2s_1+x_4f_3,\\
&x_1s_3-x_3f_2f_5+x_3f_3f_4-f_5s_1,\\
&f_1f_5^2-f_2f_3f_5+f_3^2f_4+s_2s_3
\,\big\rangle .
\end{aligned}
\]
These relations are obtained from the three defining relations for
$s_1,s_2,s_3$ after saturating with respect to the irrelevant monomial
$x_1x_2x_3x_4$.
A direct computation shows that the closed subsets
$V(I)\cap V(x_1,x_2)$ and $V(I)\cap V(x_3,x_4)$ have codimension $2$ inside
$V(I)$. Hence the codimension hypothesis of Proposition~\ref{pro:gens} is
satisfied. Therefore the Cox ring of \(X\) is the normalization of the
quotient by the saturated ideal:
\[
\mathcal R(X)
\cong
\left(
\frac{
\mathbb K[x_1,\dots,x_r,s_1,s_2,s_3]
}{
I:\langle x_1x_2x_3x_4\rangle^\infty
}
\right)^{\nu}.
\]
\end{example}

\section{Proof of Theorem~\ref{thm:2}}
\label{sec:4}

This section is devoted to the proof of Theorem~\ref{thm:2}. 
The argument is based on the birational transformations induced by primitive pairs of degree~$2$. 
We first show that such a pair gives rise to a toric conic bundle on the ambient Fano variety, and that the associated deck involution acts explicitly on the divisor class group of a general anticanonical hypersurface. 
This action is then used in the first case of the theorem: when two distinct primitive pairs of degree~$2$ are present, the corresponding involutions generate an infinite subgroup of 
\(\operatorname{Bir}(X)\), detected already on a rank-two quotient of \(\operatorname{Cl}(X)\). 
Since the birational automorphism group of a Calabi--Yau Mori dream space has only finitely many connected components, this proves that \(X\) is not a Mori dream space. 
In the second case, we analyze the two-dimensional sublattice generated by the relevant vertices. 
The corresponding toric surface section carries an elliptic fibration on a dense open subset of \(X\), and we exhibit, case by case, a section of infinite order in its Mordell--Weil group. 
This again produces infinitely many birational automorphisms and yields the desired non-Mori dream conclusion.

\subsection{Conic bundles and birational involutions}
We begin by proving preliminary lemmas about (rational) fibrations and involutions on a smooth Fano variety $P_\Delta$ induced by the presence of certain primitive pairs of $\Delta$.

\begin{lemma}\label{lem:conic-bundle}
Let $\Delta\subseteq N_{\mathbb{Q}}$ be a smooth Fano polytope and 
let $v_{1},v_{2}$ be two vertices of $\Delta$ 
such that $v_1+v_2=0$.
Then there exists a toric morphism
\[  \pi\colon\mathbb{P}_{\Delta}\longrightarrow\mathbb{P}_{\Delta_{0}},
\]
whose general fibers are conics with respect to the anticanonical degree,
where $\Delta_{0}\subseteq N_{\mathbb{Q}}/\langle v_{1}\rangle$ is the 
image of $\Delta$ under the quotient map.
\end{lemma}

\begin{proof}
Let $\Sigma = \Sigma_{\Delta}$ denote the spanning fan of the smooth Fano
polytope $\Delta \subset N_{\mathbb{Q}}$, and set
\[
  X \;:=\; \mathbb{P}_{\Delta} \;=\; X_{\Sigma},
\]
so that $X$ is a smooth projective toric variety.  
Because $v_{1}+v_{2}=0$ in $N$, the line
$\langle v_{1}\rangle_{\mathbb{Q}}$ meets $\Delta$
exactly in the segment $[v_{1},v_{2}]$.
Let
\[
  q \colon N \;\longrightarrow\;
  N' \;:=\; N / \langle v_{1}\rangle
\]
be the quotient homomorphism, and define
$\Delta_{0} := q(\Delta) \subset N'_{\mathbb{Q}}$.
The polytope $\Delta_{0}$ is still Fano (the origin is in its
interior and every vertex is primitive), but it need not be smooth or
simplicial.
Denote by $\Sigma_{0}$ the spanning fan of $\Delta_0$.
We claim that for every cone
$\sigma = \operatorname{cone}(F) \in \Sigma$
we have
\(
  q(\sigma) \subseteq\operatorname{cone}\bigl(q(F)\bigr) \in \Sigma_{0},
\)
so $q$ induces a morphism of fans
\[
  q_{*} \colon \Sigma \;\longrightarrow\; \Sigma_{0}.
\]
First of all, we consider the case in
which $v_1$ is a 
vertex of $F$. Since $\sigma$ is smooth
we can suppose that it coincides with
the positive orthant and that 
$v_1=e_n$. Let us suppose now by contradiction
that $q(\sigma)$ is not contained in a cone
of $\Sigma_0$. This implies that there exists
a vertex $v$ of $\Delta$ such that $q(v)$ lies
in the relative interior of a cone $q(\tau)$,
of dimension at least $2$,
for some face $\tau$ of $\sigma$. 
After possibly reordering we can suppose that
$\tau=\langle e_1,\dots,e_s\rangle$, for
some $2\leq s \leq n$, 
so that we can write $v=(\alpha_1,\dots,
\alpha_s,0,\dots,0,-\alpha_n)$, where $\alpha_i > 0$ 
for any $i$.
Observe that $v$ and $e_n$ cannot lie on
a cone of $\Sigma$ since otherwise the face
$\langle v,e_n\rangle$ would intersect the
face $\tau$
in its relative interior. Therefore there
is a primitive relation of degree $1$, of the 
form
$v + e_n = u$, where $u$ is another vertex of
$\Delta$, and we deduce that
\[
u=(\alpha_1,\dots,\alpha_s,0,\dots,0,1-\alpha_n).
\]
If $\alpha_n = 1$, then $u$ belongs to the positive
orthant, so that $u=e_i$, for some $i$, would contradict $s\geq 2$.
If $\alpha_n > 1$, again $u$ and $e_n$ cannot lie on
a cone of $\Sigma$, for the same reason as before. 
Repeating the same argument, after a finite number
of steps we get a contradiction.

If $-v_1$ is a vertex of $F$ we can reason in
the same way, so that in order to conclude
we have to consider the case in which 
neither $v_1$ nor $-v_1$ is a vertex of $F$. In this
case $F$ corresponds to a vertex of the
polar $Q$, lying on the slice at height
$0$. In particular $q(F)$ corresponds to a vertex of
$Q_0$, which means that $q(F)$ is a facet
of $\Sigma_0$. This proves the claim.

The morphism of fans $q_{*}$ gives a toric morphism of varieties
\[
  \pi \colon X = X_{\Sigma}
  \;\longrightarrow\;
  Y := X_{\Sigma_{0}}.
\]
The kernel lattice is
$L := \ker q = \mathbb{Z} v_{1} \simeq \mathbb{Z}$.
Inside $L_{\mathbb{Q}}$ the fan $\Sigma$ restricts to the two rays
$\rho_{1} := \mathbb{Q}_{\ge 0} v_{1}$ and
$\rho_{2} := \mathbb{Q}_{\ge 0} v_{2}$ together with the origin.
These three cones form the one–dimensional fan of
$\mathbb{P}^{1}$, so the geometric generic fibre of $\pi$ is
\[
  F \;\simeq\; \mathbb{P}^{1}.
\]
In $X$ one has
\(
  -K_{X} = \sum_{\rho \in \Sigma(1)} D_{\rho}.
\)
The fibre $F$ meets exactly the two torus–invariant divisors
$D_{v_{1}}$ and $D_{v_{2}}$, each transversely in one point, hence
\[
  (-K_{X}) \!\cdot\! F
  \;=\;
  D_{v_{1}} \!\cdot\! F
  \;+\;
  D_{v_{2}} \!\cdot\! F
  \;=\;
  1 + 1
  \;=\;
  2.
\]
Thus every fibre is a conic with respect to the anticanonical
polarisation.
The morphism
\(
  \pi \colon \mathbb{P}_{\Delta} \longrightarrow X_{\Sigma_{0}}
\)
is a \emph{toric conic bundle}: its fibres are smooth rational curves
of anticanonical degree~$2$.  
The base variety $Y = X_{\Sigma_{0}}$ is toric and projective; it is
smooth precisely when the quotient polytope $\Delta_{0}$ is simplicial,
but the conic–bundle structure exists in any case.
\end{proof}

\begin{example}[A projection that does
not induce a toric morphism]
Let $\Delta\subseteq N_\mathbb Q=\mathbb Q^{3}$ be the
lattice polytope whose vertices are the columns of the matrix
\[
  \begin{pmatrix*}[r]
      1 & 1 & 0 & -1 & 0 & 0\\
      0 & 1 & 1 & -1 & 0 & 0\\
      0 & -2 & 0 & 1 & -1 & 1
  \end{pmatrix*}.
\]
The polytope $\Delta$ is Fano, terminal and simplicial, but it is not smooth.
Write $F=\Sigma_{\Delta}$ for its spanning fan.
Its maximal cones are indexed by
\[
  \bigl\{\{1,2,3\},\;
         \{3,4,5\},\;
         \{2,3,5\},\;
         \{1,4,5\},\;
         \{1,2,5\},\;
         \{3,4,6\},\;
         \{1,4,6\},\;
         \{1,3,6\}\bigr\}.
\]

\medskip
\noindent
Choose the opposite vertices $v=(0,0,1)$ and $-v$, which form a primitive pair of degree~$2$.  Set
\[
  q\colon N_{\mathbb Q}\longrightarrow N_{\mathbb Q}/\mathbb Qv
  \qquad\text{(quotient map),}\qquad
  \Delta_{0}:=q(\Delta).
\]
Thus $q$ projects $\Delta$ along the line $\mathbb Qv$ and
$\Delta_{0}$ is the resulting quotient polytope, whose spanning fan we
denote by $F_{0}$.
The map $q$ does not extend to a toric morphism
$\mathbb P_{\!\Delta}\dashrightarrow \mathbb P_{\!\Delta_{0}}$, because
the following two three–dimensional simplicial cones of $F$
\[
  \begin{aligned}
    \sigma_{1}&=\operatorname{cone}\!\bigl\{
        (1,0,0),\ (1,1,-2),\ (0,1,0)
      \bigr\},\\[2pt]
    \sigma_{2}&=\operatorname{cone}\!\bigl\{
        (1,0,0),\ (0,1,0),\ (0,0,1)
      \bigr\},
  \end{aligned}
\]
are not mapped into any cone of~$F_{0}$.
Indeed, the maximal cones of $F_{0}$ are the four two–dimensional
simplicial cones shown below:

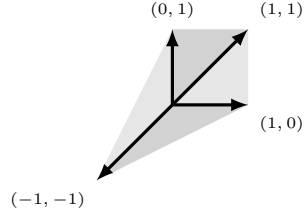
\begin{center}
\begin{tikzpicture}[>=latex,scale=1]
  \coordinate (v1) at ( 0, 1);   
  \coordinate (v2) at (-1,-1);   
  \coordinate (v3) at ( 1, 1);   
  \coordinate (v4) at ( 1, 0);   

  \fill[gray!20] (0,0) -- (v1) -- (v2) -- cycle; 
  \fill[gray!35] (0,0) -- (v1) -- (v3) -- cycle; 
  \fill[gray!20] (0,0) -- (v3) -- (v4) -- cycle; 
  \fill[gray!35] (0,0) -- (v4) -- (v2) -- cycle; 

  \draw[very thick,->] (0,0) -- (v1) node[above] {\tiny $(0,1)$};
  \draw[very thick,->] (0,0) -- (v2) node[below left] {\tiny $(-1,-1)$};
  \draw[very thick,->] (0,0) -- (v3) node[above right] {\tiny $(1,1)$};
  \draw[very thick,->] (0,0) -- (v4) node[below right] {\tiny $(1,0)$};
\end{tikzpicture}
\end{center}

Neither $q(\sigma_{1})$ nor $q(\sigma_{2})$ is contained in any of these
cones, so the projection fails to define a toric morphism.
\end{example}

\begin{lemma}\label{lem:action}
Let $\Delta\subseteq N_{\mathbb{Q}}$ be a smooth Fano polytope, let $X\in |-K_{\mathbb P_\Delta}|$ be a general smooth
anticanonical hypersurface.
Assume that $\Delta$ admits a primitive pair $\{v_{1},v_{2}\}$ of degree~$2$ and let $\sigma\in {\rm Aut}(X)$ be the involution induced by the conic bundle of Lemma~\ref{lem:conic-bundle}. Let $w_1,\dots,w_r\in{\rm Cl}(X)$ be the classes of the restrictions $D_1^X,\dots,D_r^X$ of the prime invariant divisors of $X$.
Then the action of $\sigma$ on ${\rm Cl}(X)$ is given by
\[
 \sigma(w_i) 
 = \begin{cases}
- w_2  + \displaystyle{\sum_{j\geq 3,\ j\notin \mathcal P_1}w_j}
& \text{if $i=1$}\\
  - w_1  + \displaystyle{\sum_{j\geq 3,\ j\notin \mathcal P_2}w_j}& \text{if $i=2$}\\
  w_j 
  & \text{if $v_1+v_i = v_j$ or $v_2+v_i = v_j$}\\
  w_i 
  & \text{otherwise}
 \end{cases}
\]
where $\mathcal P_j := \{i\geq 3\, :\, \{v_j,v_i\}\text{ is primitive}\}$, for
$j = 1,2$.
\end{lemma}
\begin{proof}
The hypotheses guarantee the isomorphism ${\rm Cl}(X)\simeq{\rm Cl}(\mathbb{P}_{\Delta})$.
Working in Cox coordinates
$\mathbb{K}[x_{1},\dots,x_{r}]$ ordered so that
$x_{i}$ correspond to $D_{i}$,
an anticanonical section takes the form
\[
x_{1}^{2}f_{11}+x_{1}x_{2}f_{12}+x_{2}^{2}f_{22}=0.
\]
Let $C$ be the general fiber of the morphism 
$\pi\colon\mathbb{P}_{\Delta}\longrightarrow\mathbb{P}_{\Delta_{0}}$, given in Lemma~\ref{lem:conic-bundle}.
Recall that $C$ is a conic in the anticanonical embedding and that the intersection product $D_i\cdot C$ is $1$ if $i\in\{1,2\}$ and $0$ otherwise,
by~\cite[Prop.~2.4.4]{ba}. Thus each $D_i$, with $i\geq 3$ is swept by either conics, rationally equivalent to $C$, or lines, which are components of such a conic.
In the first case $D_i$ meets both $D_1$ and $D_2$ so that $\sigma(D_i^X) = D_i^X$. In the second case $D_i$ meets only one of $D_1$, $D_2$, let us say $D_2$. Equivalently $\{v_1,v_i\}$ is a primitive pair, which must have degree $1$. Thus $v_1+v_i = v_j$ and the divisors $D_i$, $D_j$ are mapped by $\pi$ to the same prime invariant divisor of $\mathbb{P}_{\Delta_{0}}$.
In this case we deduce $\sigma(D_i^X) = D_j^X$.
It remains to compute the action of $\sigma$ on $D_1^X$ and $D_2^X$.
Observe that the divisor $D_1^X$ is cut out on $\mathbb{P}_\Delta$ by the equations $x_1 = f_{22} = 0$.
Let $i\in \mathcal P_1$, i.e.
$\{v_1,v_i\}$ is a primitive pair of degree $1$, so that $v_1 + v_i = v_j$ for some $j\geq 3$.
If $x^m$ is a monomial of $f_{22}$, then $\langle m,v_1\rangle = -1$ so that $\langle m,v_i\rangle
= \langle m,v_j\rangle + 1$. This implies that $x^m$
is divisible by $x_i$ and that, after factoring out $x_i$ we get a monomial which contains $x_i$ and $x_j$ with the same power. 
If we define
\[
g := \frac{f_{22}}{\prod_{i\in\mathcal P_1} x_i},
\]
by the above discussion $g\in\mathbb K[x_3,\dots,x_r]$
and each of its monomials is invariant under the action of $\sigma$.
Observe that $D_1^X$ is cut out by $x_1 = g = 0$ and its irreducibility implies that $g$ is irreducible. 
We conclude that $D_1^X + \sigma(D_1^X)$ is cut out on $X$ by the equation $g = 0$, which has degree 
$-K_X - 2w_2 - \sum_{i\in \mathcal P_1}w_i = w_1 - w_2 + \sum_{i\notin \mathcal P_1}w_i$. Therefore
\[
 w_1 + \sigma(w_1)
 =
 -K_X - 2w_2 - \sum_{i\in \mathcal P_1}w_i = w_1 - w_2 + \sum_{i\notin \mathcal P_1}w_i,
\]
from which we deduce the image $\sigma(w_1)$ given in the statement.
We conclude reasoning in the same way for $\sigma(w_2)$.

\end{proof}

\begin{lemma}
\label{lem:i*}
Let $\Delta\subseteq N_{\mathbb{Q}}$ be a smooth Fano polytope, let 
$\iota \colon N_F \hookrightarrow N$ 
be a sublattice
and set $\Delta_F := (N_F)_{\mathbb Q} \cap \Delta$. Then 
\[
\Delta_F^\circ=\iota^*(\Delta^\circ),
\]
where $\iota^*\colon M_{\mathbb Q}\twoheadrightarrow (M_F)_{\mathbb Q}\simeq M_{\mathbb Q}/(N_F^{\perp})_{\mathbb Q}$ is 
the dual projection,
\end{lemma}
\begin{proof}
If we denote by
$\pi \colon N \to N/N_F$ the projection
and by $\pi^*\colon (N/N_F)_{\mathbb Q}^*\hookrightarrow M_{\mathbb Q}$
the dual injection with image $(N_F^{\perp})_{\mathbb Q}$,
we have the following commutative 
diagrams
\[
\begin{tikzcd}[row sep=large, column sep=huge]
  & \Delta_F \arrow[d, hook] & \Delta \arrow[d, hook] &  & \\
  0 \arrow[r] & (N_F)_{\mathbb{Q}} \arrow[r, "\iota"] & N_{\mathbb{Q}} \arrow[r, "\pi"] &
  (N/N_F)_{\mathbb{Q}} \arrow[r] & 0\\
    0 & (M_F)_{\mathbb{Q}} \arrow[l] & M_{\mathbb{Q}} \arrow[l, "\iota^*"'] & (N_F^{\perp})_{\mathbb{Q}} \arrow[l, hook', "\pi^{*}"'] & 0.\arrow[l] \\
  & \iota^{*}(\Delta^\circ) \arrow[u, hook] & \Delta^\circ \arrow[u, hook] & &
\end{tikzcd}
\]
If $m\in\Delta^\circ$ and $u\in\Delta\cap (N_F)_{\mathbb Q}$, then 
$\langle\iota^*(m),u\rangle=\langle m,u\rangle\ge -1$, hence 
$\iota^*(m)\in\Delta_F^\circ$. This 
gives the inclusion $\iota^*(\Delta^\circ)\subseteq\Delta_F^\circ$.\\
Conversely, let $\lambda\in\Delta_F^\circ$ and $\Phi=m_0+(N_F^{\perp})_{\mathbb Q}$ be the fiber 
with $\iota^*(m_0)=\lambda$. If $\Phi\cap\Delta^\circ=\varnothing$, the separation theorem 
gives $x\in N_{\mathbb Q}$
and $\alpha\in\mathbb R$ with 
$\langle m,x\rangle<\alpha\le\langle m',x\rangle$ for all 
$m\in\Phi$, $m'\in\Delta^\circ$. 
If $x\notin (N_F)_{\mathbb Q}$ then, moving along $(N_F^{\perp})_{\mathbb Q}$, 
the values $\langle m,x\rangle$ fill all of $\mathbb R$, contradicting the first inequality. 
Hence $x\in (N_F)_{\mathbb Q}$. Moreover, since $\Delta^\circ$ is defined by 
$\langle m,y\rangle\ge -1$ for $y\in\Delta$, the separating hyperplane can be chosen as 
$\langle m,x\rangle=-1$ with $x\in\Delta\cap (N_F)_{\mathbb Q}$, so $\alpha=-1$. 
But then $\langle\lambda,x\rangle<-1$, contradicting $\lambda\in\Delta_F^\circ$. 
Thus $\Phi\cap\Delta^\circ\neq\varnothing$ and any $m^*\in\Phi\cap\Delta^\circ$ satisfies 
$\iota^*(m^*)=\lambda$. Therefore $\Delta_F^\circ\subseteq\iota^*(\Delta^\circ)$,
which concludes the proof.
\end{proof}

\subsection{The proof of Theorem~\ref{thm:2}}

\begin{proof}[Proof of Theorem~\ref{thm:2}]
We consider the two cases separately.

{\bf Case (i)}: $v_1+v_2 = v_3+v_4 = 0$.
By Lemma~\ref{lem:conic-bundle} we get two torus-invariant conic bundles on $\mathbb{P}_\Delta$; restricting to $X$ yields two degree-two morphisms
\[
\pi_{12}\colon X\longrightarrow Y_{12},\qquad \pi_{34}\colon X\longrightarrow Y_{34},
\]
and hence two involutions $\sigma_{12},\sigma_{34}\colon X\to X$ exchanging the sheets.
We are going to consider three different
cases, depending on the number of primitive
pairs involving the vertices $v_1,\dots,v_4$.
In each case considered below, both involutions preserve a subspace $W\subset {\rm Cl}(X)$, and the classes of the two conic bundles show that the quotient ${\rm Cl}(X)/W$ contains a $\sigma_{12}^*,\sigma_{34}^*$–stable rank–two subspace. Proving that the composition has infinite order on this 
plane suffices to conclude the same for ${\rm Cl}(X)$.
Therefore the image of $\operatorname{Bir}(X)$ in $\operatorname{GL}({\rm Cl}(X))$ is infinite.
By~\cite[Thm.~4.2.4.1]{adhl}, if $X$ is a Mori dream space, then the group
${\rm Bir}_2(X)$ of birational automorphisms that are isomorphisms in codimension~$1$
is an affine algebraic group. In particular, it has only finitely many irreducible
components. If $X$ is Calabi--Yau, then ${\rm Bir}(X) = {\rm Bir}_2(X)$. Hence,
whenever $X$ is a Mori dream space, the group ${\rm Bir}(X)$ has finitely many
connected components, and therefore the same is true for its image in
$\operatorname{GL}({\rm Cl}(X))$. This shows that, in each case under
consideration, $X$ cannot be a Mori dream space.
\begin{enumerate}
    \item The only primitive pairs inside $\{v_1,v_2,v_3,v_4\}$ are $\{v_1,v_2\}$ and $\{v_3,v_4\}$.
Take $W:=\langle w_5,\dots,w_r\rangle$ and work in $\langle w_3+w_4,\; w_1+w_2-w_3-w_4\rangle/W$:
\[
\sigma_{12}^* =
\begin{pmatrix*}[r]
 1 & 0 \\
 0 & -1 
 \end{pmatrix*},
\qquad
\sigma_{34}^* =
\begin{pmatrix*}[r]
 1 & 2 \\
 0 & -1\\
\end{pmatrix*}
\]
Their composition has infinite order.

\item The primitive pairs inside 
$\{v_1,v_2,v_3,v_4\}$
are $\{v_1,v_2\}$, $\{v_3,v_4\}$, and $\{v_1,v_3\}$.
After a suitable relabelling we may suppose that 
\[
   v_{1}+v_{3}=v_{5}
  \quad
   \Rightarrow
   \quad
   v_{2}+v_{5} = v_3,
   \qquad   
   v_{4}+v_{5} = v_1.
\]
Take $W:=\langle w_6,\dots,w_r\rangle$ and work in $\langle w_3+w_5,\; w_1+w_2-w_3-w_4\rangle/W$:
\[
\sigma_{12}^* =
\begin{pmatrix*}[r]
 1 & 0 \\
 0 & -1 
 \end{pmatrix*},
\qquad
\sigma_{34}^* =
\begin{pmatrix*}[r]
 1 & 1 \\
 0 & -1\\
\end{pmatrix*}.
\]
Again the composition has infinite order.

\item The primitive pairs inside 
$\{v_1,v_2,v_3,v_4\}$
are $\{v_1,v_2\}$, $\{v_3,v_4\}$, $\{v_1,v_3\}$, and $\{v_2,v_4\}$.
Relabel so that 
\[
\begin{array}{lll}   
   v_{1}+v_{3} =v_{5}
   &  \Rightarrow
   \quad
   v_{2}+v_{5} = v_3,
   & \quad   
   v_{4}+v_{5} = v_1,\\
      v_{2}+v_{4} =v_{6}
   &  \Rightarrow
   \quad
   v_{1}+v_{6} = v_4,
   & \quad   
   v_{3}+v_{6} = v_2.
\end{array}
\]
Take $W:=\langle w_7,\dots,w_r\rangle$ and work in $\langle w_3+w_4+w_5+w_6,\; w_1+w_2-w_3-w_4\rangle/W$:
\[
\sigma_{12}^* =
\begin{pmatrix*}[r]
 1 & 0 \\
 0 & -1 
 \end{pmatrix*},
\qquad
\sigma_{34}^* =
\begin{pmatrix*}[r]
 1 & 2 \\
 0 & -1\\
\end{pmatrix*}.
\]
The composition has infinite order.
\end{enumerate}
\vspace{5mm}

{\bf Case~(ii)}. $v_1 + v_2 = v_3 + v_4 - v_1 = 0$.
Let $\iota \colon N_F \hookrightarrow N$ be the sublattice generated by these four vertices, and set
\[
\Delta_F := (N_F)_{\mathbb{Q}} \cap \Delta .
\]
By construction, $\Delta_F$ contains $v_1,\dots,v_4$ among its vertices, although in principle it may have one or two additional vertices. Thus we can write $v_1,\dots,v_k$ for the vertices of $\Delta_F$, with
$4 \leq k \leq 6$.
In particular, the toric surface $\mathbb{P}_{\Delta_F}$ associated with the spanning fan $\Sigma_F$ of $\Delta_F$ is the blow-up of $\mathbb{F}_1$ at $k-4$ general points, i.e.\ a del Pezzo surface of degree $12-k$.
Since $v_1,\dots,v_k$ are also vertices of $\Delta$, the inclusion of fans $\Sigma_F \subset \Sigma$ induces an open embedding
\[
U := \mathbb{P}_{\Delta_F} \times (\mathbb{C}^*)^{n-2} \hookrightarrow \mathbb{P}_\Delta .
\]
By Lemma~\ref{lem:i*} we have
$\Delta_F^\circ = \iota^*(\Delta^\circ)$.
This does not necessarily imply that every lattice point of $\Delta_F^\circ$ is the image of a lattice point of $\Delta^\circ$, but it does imply that the vertices of $\Delta_F^\circ$ and its unique interior point (namely the origin) are.
Observe that $X \cap U$ carries an elliptic fibration
\[
  \pi \colon X \cap U \longrightarrow (\mathbb{C}^*)^{\,n-2},
\]
whose fibers are anticanonical sections of $\mathbb{P}_{\Delta_F}$, and whose zero section is cut out by $x_1 = 0$. Indeed the variable $x_1$ corresponds to the vertex $v_1$, which, on the toric surface, is a torus-invariant $(-1)$-curve, because $v_3+v_4=v_1$. By adjunction each such curve intersects a general anticanonical section in a point
that we fix as the origin of the elliptic
curve.

Our strategy is to show that we can
find a pair $(C,q)$, where $C$ is cut
on $X$ by $\mathbb P_{\Delta_F}$ and 
$q$ is a point of infinite order, which proves that for a general anticanonical hypersurface $X \subset \mathbb{P}_\Delta$ the corresponding section has infinite order in the Mordell–Weil group of the elliptic fibration $\pi$.
Consequently, the Mordell–Weil group of $\pi$ is infinite, and therefore
\[
  \operatorname{Bir}(X) 
  = 
  \operatorname{Bir}(X \cap U) 
\]
is infinite.
We now distinguish three cases, according
to the number $k$ of vertices of $\Delta_F$. 
The following picture shows the reflexive polygons $\Delta_F^\circ$. 
The labels $x_i$ indicate the Cox coordinates associated with the corresponding rays of the spanning fan of $\Delta_F$.

\begin{center}
\begin{tikzpicture}[scale=0.7]
\fill[blue!15]
    (-1,2) -- (0,1) -- (1,0) -- (2,-1) -- (0,-1) -- (-1,0) -- cycle;
\draw[line width=1pt,blue]
    (-1,2) -- (0,1) -- (1,0) -- (2,-1) -- (0,-1) -- (-1,0) -- cycle;
\foreach \P in {(-1,2), (2,-1), (0,-1), (-1,0), (0,1), (1,0), (0,0), (1,-1), (-1,1)}
    \fill[black!70] \P circle (2pt);
\node at (-1.6,0.9) {\tiny $x_4$};
\node at (0.9,0.9) {\tiny $x_2$};
\node at (1.3,-1.5) {\tiny $x_3$};
\node at (-0.8,-0.8) {\tiny $x_1$};
\node at (0,-2) {\tiny $k=4$};
\begin{scope}[xshift = 6cm]
\fill[blue!15]
    (-1,1) -- (0,1) -- (1,0) -- (2,-1) -- (0,-1) -- (-1,0) -- cycle;
\draw[line width=1pt,blue]
    (-1,1) -- (0,1) -- (1,0) -- (2,-1) -- (0,-1) -- (-1,0) -- cycle;
\foreach \P in {(2,-1), (0,-1), (-1,0), (0,1), (1,0), (0,0), (1,-1), (-1,1)}
    \fill[black!70] \P circle (2pt);
\node at (-1.4,0.6) {\tiny $x_4$};
\node at (1.2,0.4) {\tiny $x_2$};
\node at (1.2,-1.5) {\tiny $x_3$};
\node at (-0.8,-0.8) {\tiny $x_1$};
\node at (-0.5,1.3) {\tiny $x_5$};
\node at (0,-2) {\tiny $k=5$};
\end{scope}

\begin{scope}[xshift = 12cm]
\fill[blue!15]
    (-1,1) -- (0,1) -- (1,0) -- (1,0) -- (1,-1) -- (0,-1) -- (-1,0) -- cycle;
\draw[line width=1pt,blue]
    (-1,1) -- (0,1) -- (1,0) -- (1,0) -- (1,-1) -- (0,-1) -- (-1,0) -- cycle;
\foreach \P in {(0,-1), (-1,0), (0,1), (1,0), (0,0), (1,-1), (-1,1)}
    \fill[black!70] \P circle (2pt);
\node at (-1.4,0.6) {\tiny $x_4$};
\node at (0.8,0.6) {\tiny $x_2$};
\node at (0.6,-1.3) {\tiny $x_3$};
\node at (-0.8,-0.8) {\tiny $x_1$};
\node at (-0.5,1.3) {\tiny $x_5$};
\node at (1.4,-0.6) {\tiny $x_6$};
\node at (0,-2) {\tiny $k=6$};
\end{scope}
\end{tikzpicture}
\end{center}

\begin{center}
{\footnotesize The polygons $\Delta_F^\circ$.}
\end{center}

In all the cases a defining polynomial for $X$ in Cox coordinates can be written as
\[
  f = f_1 x_1^2 + f_2 x_1 x_2 + f_3 x_2^2 .
\]

\begin{enumerate}

\item If $k = 4$, then $\mathbb{P}_{\Delta_F}$ is the blow-up of $\mathbb{P}^2$ at one point. 
In this case the ratio $f_3 / x_1$ defines an element of the Cox ring.
We claim that it cuts out a point $q$
on each fiber $C$ of the fibration. 
Indeed the divisor class of $f_3 / x_1$ is $[-K - D_1 - 2D_2]$, where $K$ is the canonical class of $\mathbb{P}_{\Delta_F}$ and $D_i$ is the divisor cut out by $x_i$.
Since this divisor has intersection product $1$ with $-K$, the element $f_3 / x_1$ defines a rational section of the elliptic fibration $\pi$.
To describe these points explicitly, put
\[
x_1 = u,\quad x_2 = x_4 = 1,\quad x_3 = v.
\]
Recall that the restriction of the anticanonical linear system 
$|-K_{\mathbb P_\Delta}|$ to $\mathbb{P}_{\Delta_F}$ contains the monomials corresponding to the four vertices of $\Delta_F^\circ$ together with the interior point.
A defining equation for such a curve has the form
\[
  \alpha_0 u^2 + \alpha_1 u^2 v^3 + \alpha_2 u v + \alpha_3 v + \alpha_4 = 0 .
\]
The point $p$ is cut out by $u = 0$, hence $p = (0, -\alpha_4/\alpha_3)$.  
The point $q\in C$ is cut out by
\[
  \frac{f_3}{x_1} = \frac{\alpha_3 v + \alpha_4}{u}
  = -\alpha_0 u - \alpha_1 u v^3 - \alpha_2 v.
\]
Choosing for instance
$(\alpha_0,\alpha_1,\alpha_2,\alpha_3,\alpha_4) = (1,\,1,\,-1,\,1,\,-1)$,
the resulting curve $C$ is smooth.
With $p$ as the origin of the group law on $C$, the point $q$ has infinite order.
Indeed, over $\mathbb{Q}$ the curve is isomorphic to the elliptic curve with LMFDB label \texttt{242.b2}, whose Mordell–Weil group is isomorphic to $\mathbb{Z}$.

\item If $k = 5$, then $\mathbb{P}_{\Delta_F}$ is the blow-up of $\mathbb{P}^2$ at two points. 
To describe these points explicitly, put
\[
x_1 = u,\quad x_2 = x_3 = x_5 = 1,\quad x_4 = v.
\]
The restriction of the anticanonical linear system 
$|-K_{\mathbb P_\Delta}|$ to $\mathbb{P}_{\Delta_F}$ contains the monomials corresponding to the five vertices of $\Delta_F^\circ$ together with the interior point.
A defining equation for such a curve has the form
\[
  \alpha_0 u^2 v^3 + \alpha_1 u^2 v + \alpha_2 uv + \alpha_3 u + \alpha_4 v + \alpha_5 = 0 .
\]
The point $p$ is cut out by $u = 0$, hence $p = (0, -\alpha_5/\alpha_4)$,
while $q$ is cut out by $v = 0$, so that
$q = (-\alpha_5/\alpha_3,0)$.
Choosing for instance
$(\alpha_0,\dots,\alpha_5) = (1,\,1,\,2,\,1,\,1,\,1)$,
the resulting curve $C$ is smooth.
With $p$ as the origin of the group law on $C$, the point $q$ has infinite order.
Indeed, over $\mathbb{Q}$ the curve is isomorphic to the elliptic curve with LMFDB label \texttt{91.a1}, 
whose Mordell–Weil group is isomorphic to $\mathbb{Z}$.

\item If $k = 6$, then $\mathbb{P}_{\Delta_F}$ is the blow-up of $\mathbb{P}^2$ at three points. 
To describe these points explicitly, put
\[
x_1 = u,\quad x_2 = x_3 = x_5 = x_6 =1,\quad x_4 = v.
\]
The restriction of the anticanonical linear system 
$|-K_{\mathbb P_\Delta}|$ to $\mathbb{P}_{\Delta_F}$ contains the monomials corresponding to the five vertices of $\Delta_F^\circ$ together with the interior point.
A defining equation for such a curve has the form
\[
  \alpha_0 u^2 v^2 + \alpha_1u^2v + \alpha_2 uv^2 + \alpha_3 uv + \alpha_4 u + \alpha_5 v + \alpha_6 = 0 .
\]
The point $p$ is cut out by $u = 0$, hence $p = (0, -\alpha_6/\alpha_5)$,
while $q$ is cut out by $v = 0$, so that
$q = (-\alpha_6/\alpha_4,0)$.
Choosing for instance
$(\alpha_0,\dots,\alpha_6) = (1,\,1,\,1,\,1,\,1,\,1,\,-1)$,
the resulting curve $C$ is smooth.
With $p$ as the origin of the group law on $C$, the point $q$ has infinite order.
Indeed, over $\mathbb{Q}$ the curve is isomorphic to the elliptic curve with LMFDB label \texttt{58.a1}, 
whose Mordell–Weil group is isomorphic to $\mathbb{Z}$.
\end{enumerate}

\end{proof}

\section{Proof of Theorem~\ref{cor:1}}
\label{sec:5}

In this section we complete the proof of Theorem~\ref{cor:1}. 
The cases covered by Theorems~\ref{thm:1} and~\ref{thm:2} are first identified computationally. 
For the remaining cases, we prove that the corresponding Calabi--Yau hypersurfaces are Mori dream spaces by showing that their effective cones are rational polyhedral. 
The main ingredient is a numerical criterion ensuring that a given facet of a candidate cone lies on the boundary of the effective cone: the criterion is obtained by restricting divisor classes to suitable smooth surfaces and applying the Hodge index theorem. 
We then turn this criterion into the algorithm \textsc{TestFace}, which checks the required negativity condition by means of intersection computations in the Chow ring of the ambient toric Fano variety. 
Finally, we apply this procedure to the residual cases: after possibly enlarging the ambient effective cone by the action of the birational involutions arising from primitive pairs of degree~$2$, we verify that every facet satisfies the criterion. 
This proves the equality between the candidate cone and the effective cone of $X$, hence the Mori dream property in all remaining cases. 
The computation is illustrated explicitly for the index~$35$, while the remaining data are collected in Appendix~\ref{app:A}.

\subsection{A criterion for the effective cone}
In order to prove Theorem~\ref{cor:1}, we are going
to show that all cases not covered by 
Theorems~\ref{thm:1} and~\ref{thm:2} are Mori dream spaces. We use the fact
that a Calabi--Yau variety of dimension at most three is a Mori dream space if and only if its 
effective cone is polyhedral~\cite[Cor. 4.5]{mk}.
To establish the polyhedrality of the effective cone in our setting, we 
have the following sufficient criterion.

\begin{proposition}
\label{pro:eff}
Let $X$ be a smooth projective variety of dimension $n\geq 3$, and let
$\mathcal C\subseteq {\rm Eff}(X)$ be a rational polyhedral cone.
\begin{itemize}
    \item[(i)] 
     Let $F$
     be a facet of $\mathcal C$
     and assume there exists a smooth irreducible surface $\iota\colon S \hookrightarrow X$ such that the following hold:
\begin{enumerate}
\item every effective divisor on $X$ restricts to an effective divisor on $S$;
\item 
if we denote by $\iota^* \colon N^1(X)_\mathbb{R} \longrightarrow N^1(S)_\mathbb{R}$,
the pullback, 
the intersection form on 
$\langle\iota^*(F)\rangle$
is negative (semi)definite and non-trivial.
\end{enumerate}
Then $F$ is contained in a facet of 
${\rm Eff}(X)$.
\item[(ii)] If every facet $F$ of $\mathcal C$ satisfies $(i)$, then
$\mathcal C = {\rm Eff}(X)$.
\end{itemize}

\end{proposition}

\begin{proof}
We prove $(i)$. Let $F\subseteq \mathcal C$ be a facet and let
$\iota \colon S \hookrightarrow X$ 
be a surface as in the statement.  
By the Hodge index theorem on $S$, any nonzero negative (semi)definite 
subspace is contained in the orthogonal of some nef class.  
Thus there exists a nef class 
\[
\alpha_F \in N^{1}(S)_{\mathbb{R}}
\quad\text{such that}\quad
\langle \iota^{*}(F)\rangle \subseteq \alpha_F^\perp .
\]
Since $\alpha_F$ is nef and every effective divisor on $X$ remains effective on
$S$, the functional 
\[
\ell_F(D):=\alpha_F\cdot\iota^{*}(D)
\]
is nonnegative on ${\rm Eff}(X)$ and vanishes on $F$.  
Hence $F\subseteq {\rm Eff}(X)\cap \ell_F^\perp$, a facet of ${\rm Eff}(X)$.

We now prove $(ii)$. It is enough to 
show that if $E$ is a class that
does not belong to $\mathcal C$, then
$E$ is not effective.
Observe that, using the notation above,
if $E\notin \mathcal C$, then 
for some facet $F$ of $\mathcal C$ we have $\ell_F(E)<0$.  
But on the other hand, 
$\ell_F(D)\ge 0$ for every effective divisor class $D$, because 
$\alpha_F$ is nef and 
$\iota^{*}({\rm Eff}(X))\subseteq {\rm Eff}(S)$.  
Thus $E$ is not effective.
\end{proof}

\subsection{The algorithm \textsc{TestFace}}
We now apply the proposition above 
to smooth toric Fano fourfolds.  
Let $\Delta$ be a smooth Fano polytope of dimension $4$, let 
$\mathbb{P}_\Delta$ be the associated toric variety, and let 
$X \subseteq \mathbb{P}_\Delta$ be a smooth general anticanonical hypersurface.  
We write
\[
\mathcal R(\mathbb{P}_\Delta) = \mathbb K[x_1,\dots,x_r]
\]
and denote by $w_i \in N^1(\mathbb{P}_\Delta)_\mathbb{R}$ the class of the
torus–invariant prime divisor $\{x_i=0\}$, i.e.\ the degree of the variable
$x_i$.  
Given a suitable rational polyhedral cone
$\mathcal C  \subseteq 
\operatorname{Eff}(X)$,
we use repeatedly the following algorithm,
based on Proposition~\ref{pro:eff},
to test whether $\mathcal C$ coincides with $\operatorname{Eff}(X)$.

\vspace{4mm}

\begin{algorithm}[H]
\label{alg}
\DontPrintSemicolon
\caption{\textsc{TestFace} — Certifies whether $F$ lies on the boundary of $\mathrm{Eff}(X)$}

\KwIn{a smooth toric Fano variety $\mathbb{P}_\Delta$; a cone $F \subseteq 
\mathrm{Eff}(X)\subseteq \mathrm{Cl}(X)
=\mathrm{Cl}(\mathbb{P}_\Delta)$}
\KwOut{\textbf{true} or \textbf{false}}

\BlankLine

Compute the divisor class group $\mathrm{Cl}(\mathbb{P}_\Delta)$, the classes
$w_1,\dots,w_r \in \mathrm{Cl}(\mathbb{P}_\Delta)$ of the torus--invariant prime divisors,
primitive generators $\{\omega_1,\dots,\omega_s\}$ of the extremal rays of $F$,
the anticanonical class
$[X] := \sum_{i=1}^r w_i$,
and the movable cone
\[
\mathrm{Mov}(\mathbb{P}_\Delta)
:= \bigcap_{i=1}^r \mathrm{Cone}(w_1,\dots,\widehat{w_i},\dots,w_r).
\]
Let $\mathrm{Rays}\bigl(\mathrm{Mov}(\mathbb{P}_\Delta)\bigr)$ denote the set of primitive generators of the extremal rays of $\mathrm{Mov}(\mathbb{P}_\Delta)$.\;

\BlankLine

Set a boolean flag \emph{test} to \textbf{false}.\;

\BlankLine

\ForEach{class $A \in \mathrm{Rays}\bigl(\mathrm{Mov}(\mathbb{P}_\Delta)\bigr)$}{
  Compute the intersection matrix
  \[
    q_{A} := \bigl(A \cdot X \cdot \omega_i \cdot \omega_j\bigr)_{i,j=1}^s.
  \]\;

  \If{$ \mathrm{rank}(q_{A}) > 0$ \textbf{and} $q_{A}$ is negative semidefinite}{
    Set the boolean flag \emph{test} to \textbf{true} and \textbf{break}.\;
  }
}

\BlankLine
\Return \emph{test}.\;
\end{algorithm}

\begin{proof}[Proof of correctness of \textsc{TestFace}]
Let $F = \mathrm{Cone}(\omega_1,\dots,\omega_s)\subset {\rm  Eff}(X)$, and 
assume the algorithm returns \textbf{true}.  
Then there exists a class 
$A$ spanning an extremal ray of $\mathrm{Mov}(\mathbb{P}_\Delta)$ such that
\[
\bigl(A \cdot X \cdot \omega_i \cdot \omega_j\bigr)_{i,j=1}^s
\]
is negative (semi)definite
and non-trivial. Take a general hypersurface $Y \in |A|$ and set $S := X \cap Y$.  
Since $A \in \mathrm{Mov}(\mathbb{P}_\Delta)$, a general member of $|A|$ is smooth 
and base–point–free in codimension~$1$; hence $S$ is smooth and irreducible by 
Bertini, and every effective divisor
on $X$ restricts to an effective
divisor on $S$. Moreover,
by construction,
\[
(A\cdot X \cdot \omega_i \cdot 
\omega_j)
   \;=\;
\iota^*(\omega_i)\cdot \iota^*(\omega_j)\quad\text{on } S,
\]
so the matrix above 
is precisely the intersection form on 
on \(\langle\iota^*(F)\rangle\).
Therefore the hypotheses of  Proposition~\ref{pro:eff} (i) hold  for the facet $F$.  
\end{proof}

\subsection{Proof of Theorem~\ref{cor:1}}
\begin{proof}[Proof of Theorem~\ref{cor:1}]
We begin with the case $n=3$, corresponding to the indices $6 \le i \le 23$.
Using the Magma implementation available at
\begin{center}
\url{https://github.com/alaface/Calabi-Yau-hypersurfaces},
\end{center}
we verify that for each index $i$ in this range the associated smooth general
anticanonical hypersurface
$X \subset \mathbb{P}_{\Delta}$ satisfies the hypotheses of either
Theorem~\ref{thm:1} or Theorem~\ref{thm:2}.
Hence all cases in dimension~$3$ are completely classified by these two results,
and the corresponding indices are listed in Table~\ref{tab:34}.

We now turn to the case $n=4$, corresponding to the indices $24 \le i \le 147$.
The same script shows that all but the following cases 
\[
\{33, 34, 35, 38, 54, 93, 94, 104, 110, 117, 132, 133\}
\]
are settled directly by
Theorems~\ref{thm:1} and~\ref{thm:2}.
For each of these indices we first construct the cone $\mathcal C$ generated by the restrictions to $X$ of
the divisor classes spanning the effective cone ${\rm Eff}(\mathbb{P}_{\Delta})$.
If $X$ admits a birational involution
$\sigma$ arising from a primitive pair of degree~$2$, we enlarge $\mathcal C$ by
adding the images under $\sigma$ of its generators.
For each facet $F$ of $\mathcal C$
we verify that either $F$ itself, or its image $\sigma^*(F)$ (when the involution
$\sigma$ exists), satisfies the negativity condition required by 
Algorithm~\ref{alg}.
This is sufficient to conclude that $\mathcal C = {\rm Eff}(X)$
and thus that $X$ is a Mori dream space by~\cite[Cor. 4.5]{mk}.
In Example~\ref{ex:35}, we perform
all the computations explicitly for the index $35$.
The data for the remaining indexes are collected in Appendix~\ref{app:A}.
In this way, all cases except $i = 117$ are settled.
The latter case has been treated separately in
Example~\ref{ex:117}.
\end{proof}


\begin{example}
\label{ex:35}
In this example we illustrate in detail case $i = 35$, i.e.
$\Delta \subset \mathbb Q^4$ is the smooth Fano lattice polytope whose vertices are
the columns of
\[
\begin{pmatrix*}[r]
 1 & 0 & 0 & 0 & -1 & 0 & 0 \\
 0 & 1 & 0 & 0 & -1 & 1 & 0 \\
 0 & 0 & 1 & 0 & 0 & -1 & 0 \\
 0 & 0 & 0 & 1 & 2 & 0 & -1
\end{pmatrix*}.
\]
Let $\mathbb P_\Delta$ be the associated smooth projective toric Fano fourfold, and denote by
$D_1,\dots,D_7$ its prime torus--invariant divisors. The Cox ring is a polynomial ring with one variable for each ray, and it is graded by
${\rm Cl}(\mathbb P_\Delta)\simeq \mathbb Z^3$. The classes of the divisors $D_i$ in the
class group are encoded by the grading matrix whose $i$--th column represents the class
of $D_i$:
\[
\begin{pmatrix*}[r]
0 & 0 & 0 & 1 & 0 & 0 & 1\\
1 & 0 & 1 & 0 & 1 & 1 & 2\\
1 & 1 & 0 & 0 & 1 & 0 & 2
\end{pmatrix*}.
\]
In particular the extremal rays of the
effective cone $\rm Eff(\mathbb P_\Delta)$ are generated by the classes 
of $D_4,D_3$ and $D_2$.
Observe now that there are two primitive pairs among the vertices of $\Delta$, giving the relations
$v_3+v_6=v_2$ and $v_4+v_7=0$, of degrees $1$ and $2$ respectively. By Lemma~\ref{lem:action}, the primitive relation of degree~$2$ induces a birational
involution $\sigma\in {\rm Bir}(X)$ on a general anticanonical hypersurface
$X\subset \mathbb P_\Delta$. Via the identification
${\rm Cl}(X)\simeq {\rm Cl}(\mathbb P_\Delta)$ given by restriction, we view $\sigma^*$
as acting on divisor classes coming from $\mathbb P_\Delta$. In the basis
$(D_4,D_3,D_2)$, its action is represented by
\[
\begin{pmatrix*}[r]
 -1 & 0 & 0 \\
  2 & 1 & 0 \\
  1 & 0 & 1
\end{pmatrix*},
\]
so that $\sigma^*(D_4)=-D_4+2D_3+D_2$ 
is effective. Following the proof of Theorem~\ref{cor:1},
we consider the cone 
$\mathcal C =
\langle
D_4,\,D_3,\, D_2,\,\sigma^*(D_4)
\rangle$,
where by abuse of notation we use the same
symbol $D_i$ to denote its restriction 
to $X$.
We are now going to check that for any
facet $F$ of $\mathcal C$ there is a 
ray of the cone 
${\rm Mov}(\mathbb P_\Delta)
={\rm Cone}(D_5,D_6,D_7)$, satisfying
the conditions of Algorithm~\ref{alg}.
The facets of $\mathcal C$
can be written as
\[
\langle D_4,D_2\rangle,\ 
\sigma^*(\langle D_4,D_2\rangle), \ 
\langle D_3,D_4\rangle,\ 
\sigma^*(\langle D_3,D_4\rangle).
\]
Let us fix the first facet 
$\langle D_4,D_2\rangle$ and 
the class $D_5\in {\rm Mov}
(\mathbb P_\Delta)$. We have to
compute the matrix
\begin{equation}
\label{eq:q}
(-K_{\mathbb P_\Delta}\cdot D_5\cdot \omega_i
\cdot \omega_j)
\end{equation}
where $\omega_i,\omega_j\in\{D_4,D_2\}$.
The computations are carried out in the Chow ring of the
ambient toric fourfold, which we present directly in terms of the divisor classes $D_i$:
\[
A^\ast(\mathbb P_\Delta)\simeq
\frac{\mathbb Z[D_1,\dots,D_7]}
{\bigl(
D_1-D_5,\;
D_2-D_5+D_6,\;
D_3-D_6,\;
D_4+2D_5-D_7,\;
D_3D_6,\;
D_4D_7,\;
D_1D_2D_5
\bigr)}.
\]
Up to linear equivalence,
$D_1\sim D_5\sim D_2+D_3$, $D_6\sim D_3$ and $D_7\sim D_4+2D_3+2D_2$,
so that we obtain the isomorphic presentation
\[
A^\ast(\mathbb P_\Delta)\simeq
\frac{\mathbb Z[D_2,D_3,D_4]}
{\bigl(
D_3^2,\;
2D_2D_4+2D_3D_4+D_4^2,\;
D_2^3+2D_2^2D_3
\bigr)}.
\]
Using $D_5\sim D_2+D_3$ and
$-K_{\mathbb P_\Delta}\sim 2D_4+6D_3+5D_2$, each entry 
of~\eqref{eq:q} can be
obtained by reducing
\[
(2D_4+6D_3+5D_2)\cdot (D_2+D_3)\cdot
D_i\cdot D_j,
\qquad
i,j \in\{4,2\}
\]
in $A^4(\mathbb P_\Delta)$.
To read off the resulting integer from the reduced presentation, we normalize
$A^4(\mathbb P_\Delta)\simeq\mathbb Z$ by setting $D_1\cdot D_2\cdot D_3\cdot D_4=1$ 
(the rays $\{v_1,v_2,v_3,v_4\}$ span a smooth maximal cone). Using $D_1\sim D_2+D_3$ and $D_3^2=0$
this gives $D_2^2\cdot D_3\cdot D_4 = 1$, and one reduces further to obtain
$D_4^4=-8$. Hence, once a degree--$4$ class is reduced to a multiple of
$D_4^4$, its integral is read off immediately, and we get the matrix
\[
\begin{pmatrix*}[r]
-4 & 1\\
 1 &-2
\end{pmatrix*},
\]
which is negative definite and non-trivial, so that $D_5$ satisfies 
the conditions of 
Algorithm~\ref{alg} for 
$\langle D_4,D_2\rangle$. Therefore
$\langle D_4,D_2\rangle$ is contained
in a facet of $\rm Eff(X)$, and 
hence the same holds for the
second facet $\sigma^*(\langle
D_4,D_2\rangle)$.
If we now consider
$\langle D_3,D_4\rangle$
and the class $D_6\sim D_3
\in {\rm Mov}(\mathbb P_\Delta)$,
we obtain the matrix
\[
\begin{pmatrix*}[r]
0 & 0 \\
0 & -2
\end{pmatrix*},
\]
which satisfies again the conditions
of Algorithm~\ref{alg}, so that 
also the facets $\langle D_3,D_4\rangle$
and $\sigma^*(\langle
D_3,D_4\rangle)$
are on the boundary of $\rm Eff(X)$.
We can conclude, by means of 
Proposition~\ref{pro:eff}, that 
$\mathcal C = {\rm Eff}(X)$.
\end{example}

\begin{remark}
Let $i \colon X \hookrightarrow Z$ denote the inclusion of a smooth ample 
hypersurface in a Mori dream space $Z$.  
A natural obstruction for $X$ to be a Mori dream space is the presence of nef 
divisors on $X$ that are not semiample; equivalently, the Mori cone of curves 
on $X$ does not coincide with the pullback of the Mori cone of $Z$, and 
${\rm Nef}(X)$ is strictly larger than $i^{*}{\rm Nef}(Z)$.

For Calabi--Yau hypersurfaces in Fano manifolds of sufficiently high 
dimension, this obstruction does not occur.  
Indeed, as shown in the Appendix of~\cite{bor}, if $Z$ is a smooth Fano 
variety of dimension $\ge 4$ and $X \in |{-}K_Z|$ is smooth (the argument 
extends to the case in which $X$ has arbitrary singularities), then the map
\[
i_* \colon \mathrm{NE}(X) \longrightarrow \mathrm{NE}(Z)
\]
is an isomorphism.  
By duality it follows that
${\rm Nef}(X) = i^{*}{\rm Nef}(Z)$,
so in particular ${\rm Nef}(X)$ is rational polyhedral and every nef divisor 
on $X$ is semiample.
\end{remark}

\section{Proof of Theorem~\ref{thm:cone}}
\label{sec:6}

This section is devoted to the proof of Theorem~\ref{thm:cone}. 
The goal is to verify the Morrison--Kawamata movable cone conjecture for a family of Calabi--Yau hypersurfaces in smooth toric Fano varieties. 
We first establish an abstract criterion showing that, if two birational involutions act on the nef cone as reflections along two distinguished facets and their composition is unipotent with a one-dimensional fixed eigenspace, then the movable cone is obtained by translating the nef cone under the group generated by these involutions. 
This gives a rational polyhedral fundamental domain for the action on the effective movable cone. 
We then analyze which facets of the ambient nef cone remain on the boundary of the movable cone of the hypersurface. 
This is done by comparing the restriction of the extremal contractions of the toric Fano variety to the anticanonical hypersurface, and by translating the outcome into the language of primitive relations. 
Finally, we apply these criteria to the explicit family appearing in Theorem~\ref{thm:cone}: two primitive pairs of degree~\(2\) provide the required involutions, while the remaining extremal primitive relation gives a facet lying on the boundary of the movable cone. 
This verifies the hypotheses of the cone criterion and proves the theorem.

\subsection{A criterion from two involutions}
The Morrison--Kawamata cone conjecture~\cites{mor,kaw}, for a Calabi--Yau variety $X$,
asserts that the action of
${\rm Bir}(X)$ on the effective movable cone admits a rational polyhedral
fundamental domain. In particular, only finitely many nef chambers should
appear up to birational automorphisms, or equivalently, one expects only
finitely many isomorphism classes of minimal models of $X$ together with a
chamber decomposition of the effective movable cone governed by the action of
${\rm Bir}(X)$.
More precisely, let
\[
\overline{\rm Mov}^{e}(X):=\overline{\rm Mov}(X)\cap {\rm Eff}(X)
\]
be the effective movable cone. Since $X$ is Calabi--Yau, every birational
self-map is an isomorphism in codimension one, and therefore ${\rm Bir}(X)$
acts naturally on ${\rm Eff}(X)$, ${\rm Mov}(X)$, and
$\overline{\rm Mov}^{e}(X)$.

\begin{conjecture}[Morrison--Kawamata movable cone conjecture]
There exists a finite rational polyhedral cone $\Pi \subseteq \overline{\rm Mov}^{e}(X)$
such that
\[
\overline{\rm Mov}^{e}(X) = {\rm Bir}(X)\cdot \Pi,
\]
and such that $\Pi$ is a fundamental domain for the action of ${\rm Bir}(X)$ on
$\overline{\rm Mov}^{e}(X)$.
\end{conjecture}

This conjecture has been verified in several important families of Calabi--Yau
varieties. In the case of varieties of Wehler type, Cantat and Oguiso proved
the birational form of the cone conjecture and described the birational
automorphism group in terms of universal Coxeter groups \cite{co}. Their
approach was later extended by Y\'anez to Calabi--Yau complete intersections of
ample divisors in arbitrary products of projective spaces, where again the
movable cone conjecture is established and the structure of
${\rm Bir}(X)$ is made explicit \cite{yan}. More recently, Wang considered
Calabi--Yau hypersurfaces in Fano manifolds and proved that if $Z$ is a Fano
manifold of dimension at least $4$ whose extremal contractions are all of fiber
type, then for every smooth anticanonical hypersurface $X\in |-K_Z|$ the
movable cone conjecture holds; moreover, $X$ has a unique minimal model up to
isomorphism as an abstract variety \cite{wan}.

The examples considered in this section fit naturally into this circle of
ideas. Our goal is to show that, under suitable assumptions on the nef cone of
$X$ and on the birational involutions arising from two of its facets, the
movable cone is generated from ${\rm Nef}(X)$ by a group of pseudo-automorphisms.
This yields a rational polyhedral fundamental domain for the action of
${\rm Bir}(X)$ on $\overline{\rm Mov}^e(X)$ and therefore confirms the
Morrison--Kawamata cone conjecture in our setting.

\begin{lemma}\label{lem:closed-union-cones}
Let $\Pi\subset \mathbb R^n$ be a closed pointed polyhedral cone, and let
$G:=\langle \alpha,\beta\rangle \subset \mathrm{GL}_n(\mathbb R)$
be generated by two involutions. Assume 
that there exists a nonzero vector $e\in \Pi$ such that
$\alpha(e)=\beta(e)=e$ and that $h:=\alpha\beta$ is unipotent
with eigenspace $\mathbb Re$. Then
\[
U:=G\cdot \Pi=\bigcup_{g\in G} g(\Pi),
\]
is closed with $\partial U\subseteq G\cdot \partial \Pi$.
\end{lemma}

\begin{proof}
Since $\alpha$ and $\beta$ are involutions, the subgroup generated by $h=\alpha\beta$ has index at most $2$ in $G$, and every element of $G$ is either of the form $h^r$ or of the form $\alpha h^r$, with $r\in\mathbb Z$. Thus $U$ is the union of the cones $h^r(\Pi)$ and $\alpha h^r(\Pi)$.

We first prove that $U$ is closed. Let $(x_m)$ be a sequence in $U$ converging to some point $x\in \mathbb R^n$. If $x=0$, then $x\in U$, so there is nothing to prove. Assume $x\neq 0$.

After passing to a subsequence, we may assume that for each $m$ the point $x_m$ belongs to a cone $g_m(\Pi)$, where either the sequence $(g_m)$ is eventually constant or the cones $g_m(\Pi)$ are pairwise distinct.

If $(g_m)$ is eventually constant, say $g_m=g$ for all sufficiently large $m$, then $x_m\in g(\Pi)$ for all sufficiently large $m$. Since $g(\Pi)$ is closed, it follows that $x\in g(\Pi)\subset U$.

Assume now that the cones $g_m(\Pi)$ are pairwise distinct. Since $h$ is unipotent, we may write $h=I+N$ with $N$ nilpotent. The assumption that the $1$-eigenspace of $h$ is exactly $\mathbb Re$ means that $\ker N=\mathbb Re$. Let $v\in \mathbb R^n$ be a nonzero vector not proportional to $e$. Choose the largest integer $k\ge 0$ such that $N^k v\neq 0$. Then $N^{k+1}v=0$, and the binomial formula gives
\[
h^r(v)=\sum_{j=0}^k \binom{r}{j}N^jv .
\]
The leading term is $\binom{r}{k}N^k v$, and since $N^{k+1}v=0$, one has $N^k v\in \ker N=\mathbb Re$. Therefore the projective class $[h^r(v)]$ converges to $[e]$ as $|r|\to \infty$.

Since $\Pi$ is polyhedral, it has only finitely many extremal rays. Applying the previous observation to a set of generators of these rays, we see that the only projective accumulation point of the family $\{h^r(\Pi)\}_{r\in\mathbb Z}$ is $[e]$. The same holds for the family $\{\alpha h^r(\Pi)\}_{r\in\mathbb Z}$ because $\alpha(e)=e$. Hence $[e]$ is the only possible projective accumulation point of the family $\{g(\Pi)\}_{g\in G}$.

Now the projective classes $[x_m]$ converge to $[x]$, so necessarily $[x]=[e]$. Thus $x$ is a positive multiple of $e$. Since both generators fix $e$, every element of $G$ fixes $e$, and therefore the ray $\mathbb R_{\ge 0}e$ is contained in $U$. It follows that $x\in U$.
We have shown that $U$ is closed.

Finally, let $x\in \partial U$. Since $U$ is closed, we have $x\in U$, so there exists $g\in G$ such that $x\in g(\Pi)$. If $x\notin g(\partial \Pi)$, then $x$ lies in the interior of $g(\Pi)$. Hence some open neighborhood of $x$ is contained in $g(\Pi)$, and therefore in $U$, contradicting the fact that $x\in \partial U$. Thus $x\in g(\partial \Pi)$, and so $x\in G\cdot \partial \Pi$. This proves that $\partial U\subseteq G\cdot \partial \Pi$.
\end{proof}

\begin{proposition}\label{prop:cone-conjecture-from-two-involutions}
Let $\mathbb P_\Delta$ be a smooth Fano toric variety, and let $X\subset \mathbb P_\Delta$ be a Calabi--Yau hypersurface. Assume that there exist two facets $F_1,F_2$ of $\operatorname{Nef}(X)$ and two 
birational maps $\phi_1,\phi_2\in \operatorname{Bir}(X)$ such that:

\begin{enumerate}
\item the induced linear actions $\phi_1^*,\phi_2^*\in \operatorname{GL}(N^1(X)_{\mathbb R})$ are involutions;

\item there exists a nonzero class $e\in \operatorname{Nef}(X)$ such that $\phi_1^*(e)=\phi_2^*(e)=e$, and the element $h:=\phi_1^*\phi_2^*$ is unipotent with eigenspace $\mathbb R e$;

\item the facets $F_1$ and $F_2$ are the reflecting walls corresponding to $\phi_1^*$ and $\phi_2^*$, respectively;

\item every facet of $\operatorname{Nef}(X)$ different from $F_1$ and $F_2$ is contained in $\partial \operatorname{Mov}(X)$.
\end{enumerate}

Let $G:=\langle \phi_1^*,\phi_2^*\rangle \subset \operatorname{GL}(N^1(X)_{\mathbb R})$. Then
\[
\overline{\operatorname{Mov}}(X)=\overline{\operatorname{Mov}}^e(X)=\operatorname{Mov}(X)=G\cdot \operatorname{Nef}(X).
\]
Moreover, the image of $G$ has finite index in the image of $\operatorname{Bir}(X)$ in $\operatorname{GL}(N^1(X)_{\mathbb R})$. In particular, the cone conjecture holds for $X$.
\end{proposition}

\begin{proof}
Since $X$ is a Calabi--Yau hypersurface in a smooth Fano toric variety, one has $\operatorname{Nef}(X)=\operatorname{Nef}(\mathbb P_\Delta)$. Therefore Lemma~\ref{lem:closed-union-cones} applies to the cone $\Pi=\operatorname{Nef}(X)$ and to the involutions $\alpha=\phi_1^*$ and $\beta=\phi_2^*$.

Set $U:=G\cdot \operatorname{Nef}(X)$. By the lemma, $U$ is closed and its boundary is contained in $G\cdot \partial \operatorname{Nef}(X)$.

Since $\phi_1$ and $\phi_2$ are pseudo-automorphisms of $X$, every element of $G$ is induced by a pseudo-automorphism of $X$. Hence each cone $g(\operatorname{Nef}(X))$, with $g\in G$, is contained in $\operatorname{Mov}(X)$. It follows that $U\subseteq \operatorname{Mov}(X)$.

We claim that every boundary wall of $U$ is contained in $\partial \operatorname{Mov}(X)$. Indeed, every point of $\partial U$ lies on a translate of a facet of $\operatorname{Nef}(X)$. The two distinguished facets $F_1$ and $F_2$ do not contribute to the outer boundary of $U$, because they are interior walls of the chamber decomposition. More precisely, the cone adjacent to $\operatorname{Nef}(X)$ across $F_j$ is $\phi_j^*(\operatorname{Nef}(X))$, which is again contained in $U$. The same argument applies to every translate of these two facets: if a wall is of the form $g(F_j)$, then the two cones adjacent along it are $g(\operatorname{Nef}(X))$ and $g\phi_j^*(\operatorname{Nef}(X))$, both contained in $U$. Therefore no translate of $F_1$ or $F_2$ can occur in the outer boundary of $U$.

On the other hand, by assumption, every facet of $\operatorname{Nef}(X)$ different from $F_1$ and $F_2$ is contained in $\partial \operatorname{Mov}(X)$. Since pseudo-automorphisms preserve the movable cone, every translate under $G$ of such a facet is again contained in $\partial \operatorname{Mov}(X)$. Thus every wall which may contribute to $\partial U$ is contained in $\partial \operatorname{Mov}(X)$, and therefore $\partial U\subseteq \partial \operatorname{Mov}(X)$.

Now $U$ is a closed subset of $\operatorname{Mov}(X)$ whose boundary is contained in the boundary of $\operatorname{Mov}(X)$. Hence $U=\operatorname{Mov}(X)$. Since $U$ is closed, we also obtain $\overline{\operatorname{Mov}}(X)=\overline{\operatorname{Mov}}^e(X)=\operatorname{Mov}(X)$.

Finally, let $\Gamma$ be the image of $\operatorname{Bir}(X)$ in $\operatorname{GL}(N^1(X)_{\mathbb R})$. Every element of $\Gamma$ sends $\operatorname{Nef}(X)$ to a chamber of $\operatorname{Mov}(X)$, and by what we have proved every such chamber is of the form $g(\operatorname{Nef}(X))$ for some $g\in G$. Therefore every element of $\Gamma$ differs from an element of $G$ by an element stabilizing $\operatorname{Nef}(X)$.

The stabilizer of $\operatorname{Nef}(X)$ in $\Gamma$ is finite, because $\operatorname{Nef}(X)$ is a full-dimensional rational polyhedral cone and hence has only finitely many extremal rays, which any element of the stabilizer must permute. It follows that the image of $G$ has finite index in $\Gamma$.

Therefore the action of $\operatorname{Bir}(X)$ on $\operatorname{Mov}(X)$ admits a rational polyhedral fundamental domain, given by a finite union of translates of $\operatorname{Nef}(X)$. This proves the cone conjecture for $X$.
\end{proof}

\subsection{Facets of the nef cone and the movable cone}
In this subsection we compare the facets of the nef cone of the ambient toric
Fano variety with the boundary of the movable cone of a general anticanonical
hypersurface. The main point is to determine when a facet of
\({\rm Nef}(Z)={\rm Nef}(X)\) remains a facet of \({\rm Mov}(X)\), and when it
instead becomes an interior wall of the movable cone. We do this by analyzing
the extremal contraction associated with the dual extremal ray of the Mori
cone. The outcome is a criterion in terms of the type of contraction:
fiber type, divisorial, or small. We then translate this criterion into the
combinatorics of primitive relations of the fan, which allows us to apply it
directly to the toric Fano varieties appearing in Theorem~\ref{thm:cone}.

\begin{lemma}\label{lem:van}
Let $X$ be a smooth Fano variety, and let $D$ be a nef divisor on $X$. Assume that $D$ defines a contraction of fiber type
$\varphi\colon X \to Y$,
and that the fibers of $\varphi$ have dimension at least $2$. Then
\[
H^1(X,K_X+D)=0.
\]
\end{lemma}

\begin{proof}
Since $D$ defines the contraction $f$, it is pulled back from the base. In other words, there exists an ample divisor $A$ on $Y$ such that $D \sim \varphi^*A$. Thus the cohomology group we want to study is
$H^1(X,K_X+\varphi^*A)$.
By the projection formula 
\[
R^i \varphi_*\mathcal O_X(K_X+\varphi^*A) \cong R^i \varphi_*\mathcal O_X(K_X) \otimes \mathcal O_Y(A).
\]
The exact sequence of low-degree terms of the Leray spectral sequence for the morphism $f$ thus yields:
\[
0 \to H^1(Y, \varphi_*\mathcal O_X(K_X) \otimes \mathcal O_Y(A)) \to H^1(X, K_X+\varphi^*A) \to H^0(Y, R^1 \varphi_*\mathcal O_X(K_X) \otimes \mathcal O_Y(A)).
\]
Therefore, to show that $H^1(X, K_X+\varphi^*A)=0$, it is sufficient to prove that $\varphi_*\mathcal O_X(K_X)=0$ and $R^1 \varphi_*\mathcal O_X(K_X)=0$.

By Kollár's torsion-free theorem~\cite{ko}, the higher direct images $R^i \varphi_* \mathcal{O}_X(K_X)$ are torsion-free sheaves. Let $F$ be a general fiber of $\varphi$ of dimension $d \ge 2$. Because $X$ is smooth and Fano, and $\varphi$ is a fiber type extremal contraction, the general fiber $F$ is again a smooth Fano variety. 

By Serre duality on the smooth variety $F$, we have $H^1(F, K_F) \cong H^{d-1}(F, \mathcal O_F)^\vee$. Since $F$ is Fano, the Kodaira vanishing theorem implies $H^j(F, \mathcal O_F)=0$ for all $j>0$. Because $d \ge 2$, we get $H^1(F, K_F)=0$. Thus, the torsion-free sheaf $R^1 f_*\mathcal O_X(K_X)$ has generic rank zero, which means $R^1 \varphi_*\mathcal O_X(K_X) = 0$.
Since $F$ is Fano $H^0(F, K_F)=0$. This implies that the generic fiber of the torsion-free sheaf $\varphi_*\mathcal O_X(K_X)$ is zero, so $\varphi_*\mathcal O_X(K_X) = 0$. 
\end{proof}

\begin{proposition}\label{prop:facet-nef-to-mov}
Let $Z$ be a smooth toric Fano variety of dimension $n\geq 4$, and let
$X\in |-K_Z|$ be a smooth general member. Let
\[
\rho \colon N^1(Z)_\mathbb R \longrightarrow N^1(X)_\mathbb R
\]
be the restriction isomorphism, and let $F$ be a facet of ${\rm Nef}(Z)={\rm Nef}(X)$.
Choose a class $[D]\in \operatorname{relint}(F)$, and let
\[
\varphi:=\varphi_D \colon Z \longrightarrow Y
\]
be the contraction defined by $D$. Then $\rho(F)$ is not a facet of ${\rm Mov}(X)$ if and only if
$\varphi$ is one of the following
\begin{enumerate}
\item a contraction of fiber type with one-dimensional fibers,
\item a divisorial contraction whose exceptional divisor has one-dimensional fibers,
\item a small contraction.
\end{enumerate}
\end{proposition}

\begin{proof}
Since $[D]$ lies in the relative interior of the facet $F$, the morphism $\varphi$ contracts precisely the curves whose numerical class spans the extremal ray of ${\rm NE}(Z)$ dual to $F$.

We first establish when the morphism defined by $D|_X$ is the restriction of $\varphi$ to $X$. The exact sequence
\[
0 \longrightarrow \mathcal O_Z(K_Z+D)\longrightarrow \mathcal O_Z(D)\longrightarrow \mathcal O_X(D|_X)\longrightarrow 0
\]
shows that if $H^1(Z,K_Z+D)=0$, the restriction map $H^0(Z,D)\longrightarrow H^0(X,D|_X)$ is surjective.  By Kawamata--Viehweg vanishing theorem, this holds whenever $\varphi$ is birational, since $D$ is big. By Lemma~\ref{lem:van}, it also holds if $\varphi$ is of fiber type with fibers of dimension at least $2$. Thus, the restriction of $\varphi$ to $X$ is the morphism defined by $D_{|X}$ in all cases except possibly Case~3 below. In what follows we will denote $\varphi_{|X}$ by $\varphi_X$.

We now discuss the possible types of contractions.

\smallskip
\noindent
\emph{Case 1: $\varphi$ is of fiber type and the general fiber has dimension at least $2$.}
Let $G$ be a general fiber of $\varphi$. Since $X$ is a divisor, $\dim(G\cap X)=\dim G-1\geq 1$. Thus the general fiber of $\varphi_X$ has positive dimension, meaning $\varphi_X$ is still of fiber type. Consequently, $D|_X$ is semiample but not big, so its class lies on the boundary of ${\rm Mov}(X)$. Since $\rho(F)$ has codimension one, it is a facet of ${\rm Mov}(X)$.

\smallskip
\noindent
\emph{Case 2: $\varphi$ is divisorial and the general fibers of the exceptional divisor have dimension at least $2$.}
Let $G$ be a general fiber of ${\rm Exc}(\varphi)\to \varphi({\rm Exc}(\varphi))$. As above, $\dim(G\cap X)=\dim G-1\geq 1$. As $G$ varies, these positive-dimensional intersections sweep out a divisor in $X\cap {\rm Exc}(\varphi)$ which is contracted by $\varphi_X$. Therefore the class of $D|_X$ lies on the boundary of ${\rm Mov}(X)$, and $\rho(F)$ is a facet of ${\rm Mov}(X)$.

\smallskip
\noindent
\emph{Case 3: $\varphi$ is of fiber type with one-dimensional fibers.}
Since 
 $\varphi$ is extremal it defines  a $\mathbb P^1$-bundle onto a smooth toric Fano variety $Y$. 
By \cite[Cor.~2.4]{ca} this extremal contraction is associated to a primitive pair $v_1,v_2$ of degree two of $Z$: $v_1+v_2=0$. By Proposition~\ref{prop:facet-irr-correspondence}, the multiplicity of a general
polynomial $f$ of degree $-K_Z$ along $V(x_1,x_2)$ is~$2$, so the equation of an anticanonical hypersurface is of the form
\[
    f \;=\; f_1\,x_1^2 \;+\; f_2\,x_1 x_2 \;+\; f_3\,x_2^2.
\]
Let $\overline Z$ be the spectrum of the Cox ring of $Z$, and let $\widehat Z \subseteq \overline Z$ be its characteristic space. By the generality assumption, the zero locus
\(V(f_1,f_2,f_3)\subseteq \overline Z\)
either has codimension $3$, or else it contains an irreducible component of the form $V(x_i,x_j)$.
Since the degree of $f$ is ample, namely $-K_Z$, every component of the form $V(x_i,x_j)$ is contained in the irrelevant locus. In particular, such a component does not meet $\widehat Z$. It follows that
\(
V(f_1,f_2,f_3)\cap \widehat Z
\)
has codimension $3$ in $\widehat Z$.
Because $Z$ is smooth, the good quotient morphism
\(
p_Z\colon \widehat Z \to Z
\)
is geometric. Hence the image
\(
p_Z\bigl(V(f_1,f_2,f_3)\cap \widehat Z\bigr)
\)
has codimension $3$ in $Z$, and therefore codimension $2$ in $X$.
This image is precisely the locus over which the $\mathbb{P}^1$-bundle has fibers contained in $X$. Since this locus has codimension $2$ in $X$, it does not contain any divisor of $X$.

\smallskip
\noindent
\emph{Case 4: $\varphi$ is divisorial and the fibers of the exceptional divisor are one-dimensional.}
By \cite[Cor.~2.4]{ca} this extremal contraction $\varphi$ is associated to a primitive pair $v_1,v_2$ of degree one of $Z$: $v_1+v_2=v_3$.
The morphism $\varphi$ is the blow-up of a smooth codimension-$2$ subvariety, and its exceptional divisor $E_3$ is a $\mathbb P^1$-bundle over $\varphi(E_3)$.
We claim that the locus of fibers 
on $X\cap E_3$ has codimension at least $2$.
Indeed a defining equation for $X$ in Cox coordinates 
is $f = x_1f_1+x_2f_2$. The contraction $\varphi$ is given by $(x_1,x_2,x_3,x_4,\dots,x_r)\mapsto (x_1x_3,x_2x_3,x_4,\dots,x_r)$.
The fiber over a point on the image $V(x_1,x_2)$ of the exceptional locus 
is thus obtained by fixing the values of $x_4,\dots,x_r$ and taking $x_3=0$.
Such fiber is contained in $X$ exactly when $f_1$ and $f_2$ both vanish
for the given values of the coordinates. By the generality assumption 
on $X$ the only divisorial components of $X\cap E_3$ spanned by fibers
are of the form $E_i\cap E_3 = V(x_i,x_3)$. Reasoning as in Case 3 one
can exclude the presence of such divisors in $X$.

\smallskip
\noindent
\emph{Case 5: $\varphi$ is small.}
The exceptional locus ${\rm Exc}(\varphi)$ has codimension at least $2$ in $Z$. A general member $X$ does not contain any irreducible component of ${\rm Exc}(\varphi)$, so $X\cap {\rm Exc}(\varphi)$ has codimension at least $2$ in $X$. It follows that $\varphi_X$ does not contract any divisor, placing the class of $D|_X$ in the interior of ${\rm Mov}(X)$. Hence $\rho(F)$ is not a facet of ${\rm Mov}(X)$.

Combining the previous cases, we conclude that $\rho(F)$ is a facet of ${\rm Mov}(X)$ if and only if $\varphi$ is neither of fiber type with one-dimensional fibers, nor divisorial with one-dimensional fibers, nor small.
\end{proof}

\begin{remark}\label{rem:comb}
Let $\mathbb P_\Delta$ be a smooth toric Fano variety of dimension $n\geq 4$, and let
$X\in |-K_{\mathbb P_\Delta}|$ be a smooth general member. Let
\[
\rho \colon N^1(\mathbb P_\Delta)_\mathbb R \longrightarrow N^1(X)_\mathbb R
\]
be the restriction isomorphism, and let $F$ be a facet of ${\rm Nef}(\mathbb P_\Delta)={\rm Nef}(X)$.
The facet $F$ is orthogonal to an extremal ray $\mathbb R_{\geq 0}\cdot [C]$
of the Mori cone $NE(\mathbb P_\Delta)$. It is known that the irreducible invariant curve $C$
corresponds to a primitive relation
\[
    v_{i_1} + \dots + v_{i_k} \;=\; c_1 v_{j_1} + \dots + c_\ell v_{j_\ell},
\]
where $v_{i_1},\dots,v_{i_k},v_{j_1},\dots,w_{j_\ell}$ are vertices of $\Delta$,
and $c_i$ are positive integers. The primitive relation
is {\em extremal} if it is not positive linear combination 
of other primitive relations, or 
equivalently, if the invariant 
curve spans an extremal ray of the Mori cone.
The three cases of Proposition~\ref{prop:facet-nef-to-mov} are 
characterized in terms of primitive relations as follows.
\begin{enumerate}
    \item The first case corresponds to an extremal primitive pair $\{v_1,v_2\}$ of degree two.
\item The second case corresponds to a 
primitive pair
$\{v_1,v_2\}$ of degree one 
(which is always extremal by~\cite{bat}).
\item
The third case corresponds to an extremal primitive relation with $\ell \geq 2$.
\end{enumerate}
\end{remark}

\subsection{Proof of Theorem~\ref{thm:cone}}
\begin{proof}[Proof of Theorem~\ref{thm:cone}]
Let us describe the combinatorics of $\Delta$. 
Its vertices are (the condition $i\leq n-2$ implies that 
$v_{n}$ is not in the convex hull of $v_1,\dots,v_{n-2},v_{n+1}$):
\[
v_1=e_1,\dots,v_n=e_n,\quad
v_{n+1}=-e_1-\cdots-e_{n-2}+i e_n,\quad
v_{n+2}=-e_{n-1},\quad
v_{n+3}=-e_n.
\]
The primitive relations are therefore
\[
v_{n-1}+v_{n+2}=0,\quad
v_n+v_{n+3}=0,\quad
v_1+\cdots+v_{n-2}+v_{n+1}=i\,v_n.
\]

In particular, the Picard number of $\mathbb P_\Delta$ is equal to $3$, hence
${\rm Nef}(\mathbb P_\Delta)$ is a simplicial cone with exactly three facets.
Moreover, the first two primitive relations correspond to two primitive pairs of degree $2$,
so by Proposition~\ref{prop:facet-nef-to-mov} and Remark~\ref{rem:comb} the corresponding two facets of ${\rm Nef}(\mathbb P_\Delta)$
are not on the boundary of ${\rm Mov}(X)$ and yield the two birational involutions needed in
Proposition~\ref{prop:cone-conjecture-from-two-involutions}.
It remains to analyze the third facet. The corresponding primitive relation is
\[
v_1+\cdots+v_{n-2}+v_{n+1}=i\,v_n.
\]
This is not a primitive pair, so it does not fall under cases~(1) or~(2) of
Remark~\ref{rem:comb}. Furthermore, the right-hand side consists of a single ray,
so $\ell=1$; therefore it is not of the type described in case~(3) either.
Equivalently, by Proposition~\ref{prop:facet-nef-to-mov}, the associated contraction
is neither a fiber type contraction with one-dimensional fibers, nor a divisorial
contraction with one-dimensional fibers in the exceptional divisor, nor a small contraction.
Hence the third facet maps to a facet of ${\rm Mov}(X)$.

We have thus verified that ${\rm Nef}(\mathbb P_\Delta)={\rm Nef}(X)$ is simplicial of
Picard rank $3$, that two of its facets correspond to the required involutions, and that
the remaining facet lies on the boundary of ${\rm Mov}(X)$. Therefore all the
hypotheses of Proposition~\ref{prop:cone-conjecture-from-two-involutions} are satisfied,
and the cone conjecture follows for $X$.
\end{proof}
\begin{remark}\label{rem:cone}
When $i=0$, the toric Fano variety
$\mathbb P_\Delta$ of 
Theorem~\ref{thm:cone} is isomorphic
to $\mathbb P^1\times\mathbb P^1
\times\mathbb P^{n-2}$. In this case 
it was already known that the cone 
conjecture holds for a general
anticanonical hypersurface $X$
(see~\cite{wan}).

\end{remark}

\begin{remark}
Within the smooth toric Fano polytopes there are only $4$
with at least $7$ vertices whose extremal primitive collections have either $\ell = 0$ or 
$\ell = 1$ and $k\geq 3$.
These correspond to the indexes $[ 42, 140, 142, 143 ]$ with
the notation of the Graded Ring Database. The corresponding Fano varieties are
\[
   \mathbb P^1\times\mathbb P_{\mathbb P^2}(\mathcal O\oplus\mathcal O(2)),
   \qquad
   \mathbb P^1\times\mathbb P_{\mathbb P^2}(\mathcal O\oplus\mathcal O(1)),
   \qquad
   \mathbb P^1\times\mathbb P^1\times\mathbb P^1\times\mathbb P^1
   \qquad
   \mathbb P^1\times\mathbb P^1\times\mathbb P^2.
\]
\end{remark}

\section{K3 hypersurfaces}
\label{sec:7}

This section discusses the surface case, namely anticanonical \(K3\) hypersurfaces in smooth toric Fano threefolds. 
In this dimension the Mori dream property can be characterized in lattice-theoretic terms: a \(K3\) surface is a Mori dream space precisely when its automorphism group is finite, and for very general toric anticanonical \(K3\) surfaces this depends only on the Picard lattice. 
We use the classification of Picard lattices for these surfaces, together with Theorem~\ref{cor:1}, to determine which examples are Mori dream spaces. 
The resulting list is presented in a table containing the toric Fano threefold, the Picard lattice of a very general anticanonical member, and the corresponding Mori dream status. 
We then examine two representative examples in detail: first, a Mori dream \(K3\) surface whose Cox ring is obtained from Theorem~\ref{thm:1}; second, a non-Mori dream example with infinite automorphism group, where the automorphisms are described explicitly through involutions acting on the Picard lattice.

\subsection{The low-dimensional picture}

The case of Calabi--Yau hypersurfaces of dimension $2$, i.e. $K3$ surfaces,
is particularly convenient as a testing ground: a $K3$ surface is a Mori dream space exactly when its automorphism group is finite and this property is known to depend only on the Picard lattice of the surface  (see \cite{AHL} and references there).
The Picard lattices of anticanonical $K3$ surfaces of smooth toric Fano threefolds have been computed in \cite{Mase}.

Let $\Delta\subset N_{\mathbb{Q}}$ be a $3$-dimensional smooth Fano lattice polytope, and let
$\mathbb{P}_{\Delta}$ denote the associated smooth projective toric Fano threefold.
A very general anticanonical member
$
X\in \bigl|-K_{\mathbb{P}_{\Delta}}\bigr|
$
is a $K3$ surface.
In our setting, the question of
whether $X$ is a Mori dream surface is settled by Theorem~\ref{cor:1}: the entry
``MDS'' in the table is obtained as a direct consequence of that result. 
Moreover, Theorem \ref{thm:1} provides a presentation for the Cox rings of all Mori dream $K3$ surfaces of this type. 

The table below lists all smooth toric Fano threefolds arising from the graded ring database
of smooth toric Fano polytopes. The first column is the database index $i$ of the polytope
$\Delta$. The second column contains a matrix whose columns are the primitive generators of
the rays of the fan of $\mathbb{P}_{\Delta}$. The third column gives a standard name for
$\mathbb{P}_{\Delta}$, following the notation of \cite{WW}. The fourth column records the
Picard lattice ${\rm Cl}(X)$ of a very general surface
$X\in\bigl|-K_{\mathbb{P}_{\Delta}}\bigr|$. Finally, the last column indicates whether $X$
is a Mori dream surface.

{\footnotesize
\begin{longtable}{%
    c                                      
    l                                      
    l                                      
    >{\centering\arraybackslash\scriptsize}c 
    c}                                     
    \toprule
    \textbf{$i$} & \textbf{Vertices of $\Delta$} & \textbf{$\mathbb P_\Delta$} &
    \textbf{Picard lattice of $X$} & \textbf{MDS} \\
    \midrule
    \endfirsthead
    \toprule
    \textbf{$i$} & \textbf{Vertices of $\Delta$} & \textbf{$\mathbb P_\Delta$} &
    \textbf{Picard lattice of $X$} & \textbf{is MDS} \\
    \midrule
    \endhead
    \bottomrule
    \endfoot
    23 &
    $\begin{bmatrix*}[r]
      1&0&0&-1\\
      0&1&0&-1\\
      0&0&1&-1
    \end{bmatrix*}$ &
    $\mathbb{P}^3$ &
    $\begin{pmatrix*}[r]4\end{pmatrix*}$ &
    Yes \\[20pt]
    \midrule
    7 &
    $\begin{bmatrix*}[r]
      1&0&0&-1&0\\
      0&1&0&-1&0\\
      0&0&1& 2&-1
    \end{bmatrix*}$ &
    $\mathbb{P}\!\bigl(\mathcal{O}_{\mathbb{P}^{2}}
      \oplus \mathcal{O}_{\mathbb{P}^{2}}(2)\bigr)$ &
    $\begin{pmatrix*}[r]2&1\\1&-2\end{pmatrix*}$ &
    Yes \\[20pt]
    19 &
    $\begin{bmatrix*}[r]
      1&0&0&-1&0\\
      0&1&0& 0&-1\\
      0&0&1& 1&-1
    \end{bmatrix*}$ &
    $\mathbb{P}\!\bigl(\mathcal{O}_{\mathbb{P}^{1}}
      \oplus \mathcal{O}_{\mathbb{P}^{1}}
      \oplus \mathcal{O}_{\mathbb{P}^{1}}(1)\bigr)$ &
    $\begin{pmatrix*}[r]0&3\\3&-2\end{pmatrix*}$ &
    Yes \\[20pt]
    20 &
    $\begin{bmatrix*}[r]
      1&0&0&-1&0\\
      0&1&0& 0&-1\\
      0&0&1& 1&-1
    \end{bmatrix*}$ &
    $\mathbb{P}\!\bigl(\mathcal{O}_{\mathbb{P}^{2}}
      \oplus \mathcal{O}_{\mathbb{P}^{2}}(1)\bigr)$ &
    $\begin{pmatrix*}[r]2&2\\2&-2\end{pmatrix*}$ &
    Yes \\[20pt]
    22 &
    $\begin{bmatrix*}[r]
      1&0&0&-1&0\\
      0&1&0& 0&-1\\
      0&0&1& 0&-1
    \end{bmatrix*}$ &
    $\mathbb{P}^{2}\!\times\!\mathbb{P}^{1}$ &
    $\begin{pmatrix*}[r]0&3\\3&2\end{pmatrix*}$ &
    Yes \\[20pt]
    \midrule
    6 &
    $\begin{bmatrix*}[r]
      1&0&0&-1&0&0\\
      0&1&0&-1&1&0\\
      0&0&1& 2&-1&-1
    \end{bmatrix*}$ &
    $F_{1}^{3}$ &
    $A_{2}\oplus(4)$ &
    Yes \\[20pt]
    11 &
    $\begin{bmatrix*}[r]
      1&0&0&-1&0&0\\
      0&1&0& 0&-1&0\\
      0&0&1& 1& 1&-1
    \end{bmatrix*}$ &
    $\mathbb{P}\!\bigl(\mathcal{O}_{\mathbb{P}^{1}\!\times\!\mathbb{P}^{1}}
      \oplus \mathcal{O}_{\mathbb{P}^{1}\!\times\!\mathbb{P}^{1}}(1,1)\bigr)$ &
    $U\oplus(-12)$ &
    No \\[20pt]
    12 &
    $\begin{bmatrix*}[r]
      1&0&0&-1&0&0\\
      0&1&0& 0& 1&-1\\
      0&0&1& 1&-1& 0
    \end{bmatrix*}$ &
    $\mathbb{P}\!\bigl(\mathcal{O}_{dP_{8}}
      \oplus \mathcal{O}_{dP_{8}}(l)\bigr)$ &
    $U\oplus(-14)$ &
    No \\[20pt]
    16 &
    $\begin{bmatrix*}[r]
      1&0&0&-1& 1&-1\\
      0&1&0& 0& 0&-1\\
      0&0&1& 1&-1& 0
    \end{bmatrix*}$ &
    $F_{2}^{3}$ &
    $A_{1}\oplus\!\begin{pmatrix*}[r]0&3\\3&-2\end{pmatrix*}$ &
    Yes \\[20pt]
    17 &
    $\begin{bmatrix*}[r]
      1&0&0&-1&0&0\\
      0&1&0& 0&-1&0\\
      0&0&1& 1& 0&-1
    \end{bmatrix*}$ &
    $dP_{8}\!\times\!\mathbb{P}^{1}$ &
    $U\oplus(-16)$ &
    No \\[20pt]
    18 &
    $\begin{bmatrix*}[r]
      1&0&0&-1&0&0\\
      0&1&0& 0&0&-1\\
      0&0&1& 1&-1&-1
    \end{bmatrix*}$ &
    $\mathbb{P}\!\bigl(\mathcal{O}_{\mathbb{P}^{1}\!\times\!\mathbb{P}^{1}}
      \oplus \mathcal{O}_{\mathbb{P}^{1}\!\times\!\mathbb{P}^{1}}(1,-1)\bigr)$ &
    $U\oplus(-20)$ &
    No \\[20pt]
    21 &
    $\begin{bmatrix*}[r]
      1&0&0&-1&0&0\\
      0&1&0& 0&-1&0\\
      0&0&1& 0& 0&-1
    \end{bmatrix*}$ &
    $\mathbb{P}^{1}\!\times\!\mathbb{P}^{1}\!\times\!\mathbb{P}^{1}$ &
    $U(2)\oplus(-4)$ &
    No \\[20pt]
    \midrule
    8 &
    $\begin{bmatrix*}[r]
      1&0&0&-1&0&0&0\\
      0&1&0& 0&-1& 1&-1\\
      0&0&1& 1& 1&-1& 0
    \end{bmatrix*}$ &
    $F_{3}^{4}$ &
    $U\oplus\!\begin{pmatrix*}[r]-2&1\\1&-6\end{pmatrix*}$ &
    No \\[20pt]
    10 &
    $\begin{bmatrix*}[r]
      1&0&0&-1&0&0&0\\
      0&1&0& 0&-1& 1& 0\\
      0&0&1& 1& 1&-1&-1
    \end{bmatrix*}$ &
    $F_{2}^{4}$ &
    $U\oplus\!\begin{pmatrix*}[r]-4&2\\2&-6\end{pmatrix*}$ &
    No \\[20pt]
    13 &
    $\begin{bmatrix*}[r]
      1&0&0&-1&0&0&0\\
      0&1&0& 0& 1&-1& 0\\
      0&0&1& 1&-1& 0&-1
    \end{bmatrix*}$ &
    $F_{1}^{4}$ &
    $U\oplus\!\begin{pmatrix*}[r]-4&1\\1&-8\end{pmatrix*}$ &
    No \\[20pt]
    14 &
    $\begin{bmatrix*}[r]
      1&0&0&-1& 1&-1&0\\
      0&1&0& 0& 0& 0&-1\\
      0&0&1& 1&-1& 0&0
    \end{bmatrix*}$ &
    $dP_{7}\!\times\!\mathbb{P}^{1}$ &
    $U\oplus\!\begin{pmatrix*}[r]-4&2\\2&-8\end{pmatrix*}$ &
    No \\[20pt]
    \midrule
    9 &
    $\begin{bmatrix*}[r]
      1&0&0&-1&0&0&0&0\\
      0&1&0& 0&-1& 1&-1&0\\
      0&0&1& 1& 1&-1&0&-1
    \end{bmatrix*}$ &
    $F_{1}^{5}$ &
    $U\oplus\!\begin{pmatrix*}[r]
        -4& 1& 2\\
         1&-4& 0\\
         2& 0&-4
      \end{pmatrix*}$ &
    No \\[20pt]
    15 &
    $\begin{bmatrix*}[r]
      1&0&0&-1& 1&-1&0&0\\
      0&1&0& 0& 0& 0&-1&0\\
      0&0&1& 1&-1& 0&0&-1
    \end{bmatrix*}$ &
    $dP_{6}\!\times\!\mathbb{P}^{1}$ &
    $U\oplus\!\bigl(\begin{smallmatrix}-4&2\\2&-4\end{smallmatrix}\bigr)\!\oplus\!(-4)$ &
    No \\[20pt]
\end{longtable}
}


\subsection{Examples}
We now discuss exemplarily two cases: one Mori dream and another with infinite automorphism group.

\begin{example}
Let $\mathbb P_\Delta$ be the smooth toric fano threefold defined by the polytope
of index $7$ in the Graded Ring Database.
The Cox ring of $\mathbb P_\Delta$ is $\mathbb{K}[x_1,\dots,x_5]$ with grading matrix
\[
\begin{bmatrix*}[r]
    0 & 0 & 1 & 0 & 1\\
    1 & 1 & 0 & 1 & 2
    \end{bmatrix*}.
\]
Moreover ${\rm Eff}(\mathbb P_\Delta)={\rm Cone}(e_1,e_2)$ and 
${\rm Nef}(X)={\rm Cone}(e_1+2e_2,e_2)$
Observe that the map induced by $(x_1,\dots,x_5)\mapsto (x_1,x_2,x_4)$ is the natural projection $\pi:\mathbb P_\Delta\to \mathbb P^2$.
The equation of an anticanonical hypersurface is of the form
\[
f_0(x_1,x_2,x_4)x_3^2-f_1(x_1,x_2,x_4)x_3x_5+f_2(x_1,x_2,x_4)x_5^2=0,
\]
where $f_0$ has degree $5e_2$, $f_1$ has degree $3e_2$ and $f_2$ has degree $e_2$.
The restriction of   $\pi$ to a general hypersurface $X$ in this linear system defines a double cover $X\to \mathbb P^2$ branched along the smooth sextic curve  of equation 
\[
f_1^2-4f_0f_2=0.
\]
The line in $\mathbb P^2$ defined by $f_2=0$ is a $3$-tangent line to $B$ which splits in the double cover, giving two $(-2)$-curves $R_1, R_2$ of $S$ intersecting at three points, defined by $x_3=0$ and $f_2=f_0x_3-f_1x_5=0$. 
The classes $h:=i^*(e_2)$ and $e:=i^*(e_1)$  generate a sublattice of ${\rm Cl}(X)$  isometric to the lattice given in the last column of the table. The classes of the two $(-2)$-curves, which generate the effective cone of $S$, are $[R_1]=e$ and $[R_2]=h-e$.
The preimage in $X$ of the plane quintic defined by $f_0=0$ also splits in the union of two genus six curves $C_1,C_2$ defined by $x_5=0$ and $f_0=f_1x_3+f_2x_5=0$, whose classes in ${\rm Cl}(X)$ are $[C_1]=2h+e$ and $[C_2]=3h-e$.
 
By Theorem \ref{thm:1}, since the ambient toric variety has a unique primitive collection of cardinality two with $d=2$, defined by $x_3=x_5=0$, the Cox ring of $X$ is finitely generated and given by
\[
\mathcal R(X)=\frac{\mathbb{K}[x_1,\dots,x_5,s_1,s_2]}{\langle f_0+x_5s_2,f_1+x_3s_2+x_5s_1, f_2+x_3s_1  \rangle},
\]
where $x_1,\dots,x_5$ are the restrictions of the generators of $\mathcal R(\mathbb P)$, 
$s_1$ defines the $(-2)$-curve $R_2$  and $s_2$ defines the genus six curve $C_2$.
\end{example}

\begin{example}
Let \(\mathbb P_\Delta\) be the smooth toric Fano threefold corresponding to
the polytope labelled \(11\) in the Graded Ring Database. Its Cox ring is
\(\mathbb K[x_1,\ldots,x_6]\), with grading matrix
\[
\begin{bmatrix*}[r]
    0 & 0 & 1 & 0 & 0 & 1\\
    0 & 1 & 0 & 0 & 1 & 1\\
    1 & 0 & 0 & 1 & 0 & 1
\end{bmatrix*}.
\]
Thus \(\deg x_1=\deg x_4=e_3\), \(\deg x_2=\deg x_5=e_2\),
\(\deg x_3=e_1\), and \(\deg x_6=e_1+e_2+e_3\). The irrelevant ideal is
\((x_1,x_4)\cap (x_2,x_5)\cap (x_3,x_6)\). Moreover,
\(\operatorname{Eff}(\mathbb P_\Delta)=\operatorname{Cone}(e_1,e_2,e_3)\),
whereas
\(\operatorname{Nef}(\mathbb P_\Delta)=
\operatorname{Cone}(e_2,e_3,e_1+e_2+e_3)\).

The projection induced by
\((x_1,\ldots,x_6)\mapsto (x_1,x_4,x_2,x_5)\) defines a morphism
\(\pi:\mathbb P_\Delta\to \mathbb P^1\times \mathbb P^1\). Since
\(-K_{\mathbb P_\Delta}=2e_1+3e_2+3e_3\), an anticanonical hypersurface has
equation
\[
f_0x_3^2+f_1x_3x_6+f_2x_6^2=0,
\]
where \(f_0,f_1,f_2\) have degrees \((0,3,3)\), \((0,2,2)\), and \((0,1,1)\),
respectively. For a general choice of \(f_i\), the restriction
\(\pi_{|X}:X\to \mathbb P^1\times \mathbb P^1\) is a double cover branched
along the smooth curve of bidegree \((4,4)\) defined by \(f_1^2-4f_0f_2=0\).

The two projections \(p_i\circ \pi\) restrict to elliptic fibrations
\(\alpha_i:X\to \mathbb P^1\). We denote by \(F_i\) the class of a fibre of
\(\alpha_i\), with the convention \(F_1=e_2\) and \(F_2=e_3\). The curve
\(S=X\cap\{x_3=0\}\) is a smooth rational curve and is a section of both
elliptic fibrations. For \(X\) very general in this family, we have
\(\operatorname{Pic}(X)=\mathbb Z e_1\oplus \mathbb Z e_2\oplus \mathbb Z e_3\),
where \(e_1=[S]\). With respect to this basis, the intersection form is
\[
Q=
\begin{pmatrix}
-2 & 1 & 1\\
1 & 0 & 2\\
1 & 2 & 0
\end{pmatrix}.
\]
The form \(Q\) is equivalent to \(U\oplus \langle -12\rangle\). Indeed, the
change of basis
\((e_1,e_2,e_3)\mapsto (e_2,e_1+e_2,2e_1+5e_2-e_3)\) gives the standard
hyperbolic plane on the first two generators and a vector of square \(-12\)
orthogonal to it. In particular,
\(A_{\operatorname{Pic}(X)}\simeq \mathbb Z/12\mathbb Z\).

Let \(\sigma\) be the involution associated with the double cover
\(\pi_{|X}\), and let \(\iota_i\) be the involution induced by fibrewise
inversion on \(\alpha_i\), with zero section \(S\). Their actions on
\(\operatorname{Pic}(X)\), with respect to the basis \(e_1,e_2,e_3\), are
\[
\sigma^*=
\begin{pmatrix}
-1 & 0 & 0\\
1 & 1 & 0\\
1 & 0 & 1
\end{pmatrix},
\qquad
\iota_1^*=
\begin{pmatrix}
1 & 0 & 4\\
0 & 1 & 10\\
0 & 0 & -1
\end{pmatrix},
\qquad
\iota_2^*=
\begin{pmatrix}
1 & 4 & 0\\
0 & -1 & 0\\
0 & 10 & 1
\end{pmatrix}.
\]
Let \(G=\langle \sigma,\iota_1,\iota_2\rangle\). We will prove that
\(\operatorname{Aut}(X)=G\).

We first determine a fundamental region for the action of \(G\) on the nef cone.
Let \(D=xe_1+ye_2+ze_3\). Then
\(D^2=-2x^2+2xy+2xz+4yz\). We take the component \(\mathcal C^+\) of the
positive cone containing the nef cone. Since \(h=e_1+e_2+e_3\) is nef and
\(D\cdot h=3(y+z)\), this component is characterized by \(D^2>0\) and
\(y+z>0\).

The fixed hyperplanes of \(\sigma^*,\iota_2^*,\iota_1^*\) are respectively
\(x=0\), \(y=0\), and \(z=0\). Consider the chamber
\[
\Delta=\{D\in \mathcal C^+:x\geq 0,\ y\geq 0,\ z\geq 0\}.
\]
The ping-pong argument \cite[Ch.II.B]{pingpong} applies to the three subsets defined by \(x<0\),
\(y<0\), and \(z<0\). These subsets are pairwise disjoint inside
\(\mathcal C^+\): for instance, \(y\leq 0\) and \(z\leq 0\) contradict
\(y+z>0\), while the cases involving \(x\leq 0\) give \(D^2\leq 0\). Moreover,
each involution sends the union of the other two subsets into its own subset.
Therefore
\[
\langle \sigma^*,\iota_1^*,\iota_2^*\rangle
\simeq
\mathbb Z/2\mathbb Z*
\mathbb Z/2\mathbb Z*
\mathbb Z/2\mathbb Z.
\]

The chamber \(\Delta\) is a fundamental chamber for the action of \(G\) on the
positive cone. Indeed, if an integral class \(D\in\mathcal C^+\) has a negative
coordinate, applying the corresponding involution strictly decreases the
positive integer \(D\cdot h\). Thus the process terminates and moves \(D\) into
\(\Delta\).

Intersecting with the nef cone gives the following fundamental region for the
action of \(G\) on \(\operatorname{Nef}(X)\):
\[
\Pi=
\Delta\cap \operatorname{Nef}(X)
=
\{xe_1+ye_2+ze_3:x,y,z\geq 0,\ y+z\geq 2x\}.
\]
Indeed, inside \(\Delta\), nefness imposes \(D\cdot S=D\cdot e_1=-2x+y+z\geq 0\).
Conversely, the extremal rays of the cone on the right are
\(e_2,e_3,e_1+2e_2,e_1+2e_3\), and all these classes are nef.

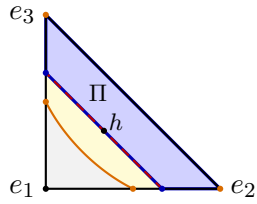
\begin{center}
\begin{tikzpicture}[scale=2.3, line join=round, line cap=round]

  \fill[gray!10] (0,0) -- (1,0) -- (0,1) -- cycle;

  \fill[yellow!20]
    (0,1/2) --
    (0,1) --
    (1,0) --
    (1/2,0) --
    plot[domain=0.5:0, samples=160, smooth]
      (\x,{(3-2*\x - sqrt(1+12*\x-12*\x*\x))/4})
    -- cycle;

  \draw[orange!85!black, thick]
    plot[domain=0:0.5, samples=160, smooth]
      (\x,{(3-2*\x - sqrt(1+12*\x-12*\x*\x))/4});

  \fill[blue!18]
    (1,0) -- (2/3,0) -- (0,2/3) -- (0,1) -- cycle;

  \draw[blue!70!black, very thick]
    (1,0) -- (2/3,0) -- (0,2/3) -- (0,1) -- cycle;

  \draw[red!75!black, dashed, thick]
    (2/3,0) -- (0,2/3);

  \draw[black, thick]
    (0,0) -- (1,0) -- (0,1) -- cycle;

  \fill[orange!85!black] (1/2,0) circle (0.5pt);
  \fill[orange!85!black] (0,1/2) circle (0.5pt);

  \fill[black] (1/3,1/3) circle (0.5pt);
  \node[right] at (1/3-0.03,1/3+0.05) {\footnotesize{$h$}};

  \fill[blue!70!black] (2/3,0) circle (0.5pt);
  \fill[blue!70!black] (0,2/3) circle (0.5pt);

  \fill (0,0) circle (0.5pt);
  \fill[orange!85!black] (1,0) circle (0.5pt);
  \fill[orange!85!black] (0,1) circle (0.5pt);

  \node[left] at (0,0) {$e_1$};
  \node[right] at (1,0) {$e_2$};
  \node[left] at (0,1) {$e_3$};

  \node[black!70!black, align=center] at (0.3,0.55)
    {\footnotesize{$\Pi$}};
\end{tikzpicture}
\end{center}

We now choose a distinguished class in \(\Pi\). We claim that
\(h=e_1+e_2+e_3\) is the unique integral class \(D\in\Pi\) satisfying
\(D^2=6\) and \(\operatorname{div}(D)=3\), where
\(\operatorname{div}(D)=\gcd\{D\cdot L:L\in\operatorname{Pic}(X)\}\).
Indeed, if \(D=xe_1+ye_2+ze_3\in\Pi\) and \(D^2=6\), then
\(-x^2+x(y+z)+2yz=3\). Since \(y+z\geq 2x\), one gets \(x^2\leq 3\), hence
\(x=1\). The equation becomes \(y+z+2yz=4\), whose non-negative integral
solutions are \((y,z)=(4,0),(1,1),(0,4)\). The corresponding classes are
\(e_1+4e_2\), \(h\), and \(e_1+4e_3\), and only \(h\) has divisibility \(3\).

We now analyze the possible automorphisms fixing \(h\). Observe that
\(h^\perp=\langle e_1,e_2-e_3\rangle\), with diagonal form
\(\operatorname{diag}(-2,-4)\). Therefore \(\pm e_1\) are the only classes of
self-intersection \(-2\) in \(h^\perp\). Thus any isometry fixing \(h\) maps
\(e_1\) to \(\pm e_1\). If the isometry is induced by an automorphism of \(X\),
then \(e_1=[S]\) must be mapped to an effective class. Hence \(e_1\) is mapped
to \(e_1\), since \(-e_1\) is not effective.

An explicit computation shows that the only isometries of
\(\operatorname{Pic}(X)\) fixing both \(h\) and \(e_1\) are the identity and the
involution exchanging \(e_2\) and \(e_3\), namely
\[
\tau=
\begin{pmatrix}
1 & 0 & 0\\
0 & 0 & 1\\
0 & 1 & 0
\end{pmatrix}.
\]
We now show that \(\tau\) is not induced by an automorphism of \(X\). A generator
of the discriminant group \(A_{\operatorname{Pic}(X)}\simeq \mathbb Z/12\mathbb Z\)
is represented by \(\gamma=(2e_1-e_2+5e_3)/12\), and one checks that
\(\tau(\gamma)\equiv 7\gamma\). Thus \(\tau\) does not act as
\(\pm\operatorname{id}\) on the discriminant group.

On the other hand, since \(\operatorname{rk} T(X)=19\) is odd, the action of an
automorphism of \(X\) on \(T(X)\) is \(\pm\operatorname{id}\) by \cite{ni}.
The actions on the discriminant groups of \(\operatorname{Pic}(X)\) and \(T(X)\)
must be compatible. Hence the action induced on \(A_{\operatorname{Pic}(X)}\)
must also be \(\pm\operatorname{id}\), which excludes \(\tau\).

We also use that the natural representation
\(\operatorname{Aut}(X)\to \operatorname{O}(\operatorname{Pic}(X))\) is
injective. Indeed, if an automorphism acts trivially on \(\operatorname{Pic}(X)\),
then it acts on \(T(X)\) as \(\pm\operatorname{id}\). The case
\(+\operatorname{id}\) gives the identity by the global Torelli theorem. The
case \(-\operatorname{id}\) would force \(A_{T(X)}\) to be \(2\)-elementary,
contradicting
\(A_{T(X)}\simeq A_{\operatorname{Pic}(X)}\simeq \mathbb Z/12\mathbb Z\).

We can now prove that \(\operatorname{Aut}(X)=G\). Let
\(f\in\operatorname{Aut}(X)\). Then \(f^*h\) is nef, integral, has square \(6\),
and has divisibility \(3\). Since \(\Pi\) is a fundamental region for the action
of \(G\) on the nef cone, there exists \(g\in G\) such that
\((f\circ g)^*h\in\Pi\). By the uniqueness of \(h\) inside \(\Pi\), we get
\((f\circ g)^*h=h\). By the previous paragraph, the action of \(f\circ g\) on
\(\operatorname{Pic}(X)\) must be trivial. By injectivity,
\(f\circ g=\operatorname{id}\), and therefore \(f\in G\). Hence
\[
\operatorname{Aut}(X)
=
G
\simeq
\mathbb Z/2\mathbb Z*
\mathbb Z/2\mathbb Z*
\mathbb Z/2\mathbb Z.
\]
\end{example}

\begin{remark}
The arguments in the previous example are inspired by the proof of \cite[Theorem 1.7]{oguiso}, where Oguiso computes the automorphism group of a very general smooth quartic surface containing two disjoint lines. These K3 surfaces admit an alternative realization as anticanonical hypersurfaces of the smooth toric Fano threefold associated with the polytope 18. 
\end{remark}

\appendix

\section*{Appendix A. Data for the remaining fourfold cases}
\addcontentsline{toc}{section}{Appendix A. Data for the remaining fourfold cases}
\label{app:A}

In this appendix we collect, for each of the smooth toric Fano fourfolds with index
\[
i\in \{33,\allowbreak 34,\allowbreak 35,\allowbreak 38,\allowbreak 54,\allowbreak 93,\allowbreak 94,\allowbreak 104,\allowbreak 110,\allowbreak 132,\allowbreak 133\},
\]
the concrete data needed to carry out the argument used in the proof of Theorem~\ref{cor:1}. More precisely, the table records the Chow ring $A^*(\mathbb P_\Delta)$ of the ambient toric variety, the generators of the cone $\mathcal C \subseteq {\rm Eff}(X)$ for a general smooth anticanonical hypersurface $X\subseteq \mathbb P_{\Delta}$, and suitable divisor classes $D\in {\rm Mov}(\mathbb P_\Delta)$ that allow one to test all facets of $\mathcal C$ via Algorithm~\ref{alg}.

\vspace{5mm}


{\footnotesize
\setlength{\parindent}{0pt}
\setlength{\tabcolsep}{9pt}
\renewcommand{\arraystretch}{1.25}

\begin{longtable}{@{}r l p{0.78\linewidth}@{}}
\midrule

33 & $A^*(\mathbb{P}_\Delta)$ &
$\displaystyle
\frac{\mathbb{Q}[D_1,\dots,D_8]}{
\left(
\begin{array}{l}
D_1-D_5,\; D_2-D_5+D_6+D_7,\; D_3-D_6,\; D_4+2D_5-D_7-D_8,\\
D_2D_8,\; D_3D_6,\; D_4D_7,\; D_4D_8,\; D_1D_2D_5,\; D_1D_5D_7
\end{array}
\right)}$\\[2mm]
 & $\mathrm{Eff}(X)$ &
$\operatorname{Cone}\big(3D_5+D_6-D_7-D_8,\; D_2,\; D_6,\; D_7,\; D_4\big)$\\
 & $D$ &
$\{\,D_6,\; D_5\,\}$\\
\addlinespace[3mm]
\midrule
34 & $A^*(\mathbb{P}_\Delta)$ &
$\displaystyle
\frac{\mathbb{Q}[D_1,\dots,D_8]}{
\left(
\begin{array}{l}
D_1-D_5+D_7,\; D_2-D_5+D_6,\; D_3-D_6,\; D_4+2D_5-D_7-D_8,\\
D_1D_8,\; D_3D_6,\; D_4D_7,\; D_4D_8,\; D_1D_2D_5,\; D_2D_5D_7
\end{array}
\right)}$\\[2mm]
 & $\mathrm{Eff}(X)$ &
$\operatorname{Cone}\big(3D_5+D_6-D_7-D_8,\; D_1,\; D_2,\; D_6,\; D_7,\; D_4\big)$\\
 & $D$ &
$\{\,D_6,\; D_2\,\}$\\
\addlinespace[3mm]
\midrule
35 & $A^*(\mathbb{P}_\Delta)$ &
$\displaystyle
\frac{\mathbb{Q}[D_1,\dots,D_7]}{
\left(
\begin{array}{l}
D_1-D_5,\; D_2-D_5+D_6,\; D_3-D_6,\; D_4+2D_5-D_7,\;
D_3D_6,\; D_4D_7,\; D_1D_2D_5
\end{array}
\right)}$\\[2mm]
 & $\mathrm{Eff}(X)$ &
$\operatorname{Cone}\big(
D_2,\; D_3,\; D_4,\; D_2+2D_3-D_4\big)$\\
 & $D$ &
$\{\,D_5,\; D_6\,\}$\\
\addlinespace[3mm]
\midrule
38 & $A^*(\mathbb{P}_\Delta)$ &
$\displaystyle
\frac{\mathbb{Q}[D_1,\dots,D_8]}{
\left(
\begin{array}{l}
D_1-D_5,\; D_2-D_5+D_6+D_8,\; D_3-D_8,\; D_4+2D_5-D_6-D_7-D_8,\\
D_2D_7,\; D_3D_8,\; D_4D_6,\; D_4D_7,\; D_1D_2D_5,\; D_1D_5D_6
\end{array}
\right)}$\\[2mm]
 & $\mathrm{Eff}(X)$ &
$\operatorname{Cone}\big(3D_5-D_6-D_7+D_8,\; D_2,\; D_8,\; D_6,\; D_4\big)$\\
 & $D$ &
$\{\,D_8,\; D_5\,\}$\\
\addlinespace[3mm]
\midrule
54 & $A^*(\mathbb{P}_\Delta)$ &
$\displaystyle
\frac{\mathbb{Q}[D_1,\dots,D_9]}{
\left(
\begin{array}{l}
D_1-D_5+D_7-D_8,\; D_2-D_5+D_6-D_8,\; D_3+D_5-D_6-D_9,\; D_4+D_5-D_7-D_9,\\
D_3D_6,\; D_4D_7,\; D_5D_9,\; D_1D_2D_5,\; D_1D_2D_8,\; D_1D_2D_9,\\
D_1D_3D_8,\; D_1D_3D_9,\; D_1D_5D_6,\; D_1D_6D_8,\; D_2D_4D_8,\; D_2D_4D_9,\\
D_2D_5D_7,\; D_2D_7D_8,\; D_3D_4D_8,\; D_3D_4D_9,\; D_5D_6D_7,\; D_6D_7D_8
\end{array}
\right)}$\\[2mm]
 & $\mathrm{Eff}(X)$ &
$\operatorname{Cone}\big(D_1,\;D_2,\;D_3,\;D_4,\;D_5,\;D_6,\;D_7,\;D_8,\;D_9\big)$\\
 & $D$ &
$\{\,D_7+D_9,\; D_4+D_6,\; D_8+D_9,\; D_5+D_8,\; D_6+D_9\,\}$\\
\addlinespace[3mm]
\midrule
93 & $A^*(\mathbb{P}_\Delta)$ &
$\displaystyle
\frac{\mathbb{Q}[D_1,\dots,D_8]}{
\left(
\begin{array}{l}
D_1-D_5,\; D_2-D_6+D_7-D_8,\; D_3-D_8,\; D_4+D_5+D_6-D_7,\\
D_1D_5,\; D_2D_6,\; D_4D_7,\; D_6D_7,\; D_2D_3D_8,\; D_3D_4D_8
\end{array}
\right)}$\\[2mm]
 & $\mathrm{Eff}(X)$ &
$\operatorname{Cone}\big(D_5-D_6+2D_8,\; D_5,\; D_2,\; D_4,\; D_6\big)$\\
 & $D$ &
$\{\,D_5,\; D_8\,\}$\\
\addlinespace[3mm]
\midrule
94 & $A^*(\mathbb{P}_\Delta)$ &
$\displaystyle
\frac{\mathbb{Q}[D_1,\dots,D_8]}{
\left(
\begin{array}{l}
D_1-D_5,\; D_2-D_6+D_7,\; D_3-D_8,\; D_4+D_5+D_6-D_7-D_8,\\
D_1D_5,\; D_2D_6,\; D_4D_7,\; D_6D_7,\; D_2D_3D_8,\; D_3D_4D_8
\end{array}
\right)}$\\[2mm]
 & $\mathrm{Eff}(X)$ &
$\operatorname{Cone}\big(2D_5-D_7+2D_8,\; D_5,\; D_2,\; D_4,\; D_7\big)$\\
 & $D$ &
$\{\,D_5-D_6+3D_8,\; D_5,\; D_2\,\}$\\
\addlinespace[3mm]
\midrule
104 & $A^*(\mathbb{P}_\Delta)$ &
$\displaystyle
\frac{\mathbb{Q}[D_1,\dots,D_8]}{
\left(
\begin{array}{l}
D_1-D_5,\; D_2-D_6+D_7-D_8,\; D_3+D_6-D_7,\; D_4+D_5-D_8,\\
D_1D_5,\; D_2D_6,\; D_3D_7,\; D_6D_7,\; D_2D_4D_8,\; D_3D_4D_8
\end{array}
\right)}$\\[2mm]
 & $\mathrm{Eff}(X)$ &
$\operatorname{Cone}\big(D_5-D_6+2D_8,\; D_5,\; D_2,\; D_3,\; D_4,\; D_6\big)$\\
 & $D$ &
$\{\,D_6+D_8,\; D_5,\; D_5-D_7+3D_8,\; D_7\,\}$\\
\addlinespace[3mm]
\midrule
110 & $A^*(\mathbb{P}_\Delta)$ &
$\displaystyle
\frac{\mathbb{Q}[D_1,\dots,D_8]}{
\left(
\begin{array}{l}
D_1-D_5+D_7-D_8,\; D_2-D_8,\; D_3+D_6-D_7,\; D_4+D_5-D_6,\\
D_1D_5,\; D_3D_7,\; D_4D_6,\; D_1D_2D_8,\; D_2D_3D_8,\; D_2D_4D_8,\; D_5D_6D_7
\end{array}
\right)}$\\[2mm]
 & $\mathrm{Eff}(X)$ &
$\operatorname{Cone}\big(D_1,\; D_3,\; D_4,\; D_5\big)$\\
 & $D$ &
$\{\,D_8,\; D_6\,\}$\\
\addlinespace[3mm]
\midrule
132 & $A^*(\mathbb{P}_\Delta)$ &
$\displaystyle
\frac{\mathbb{Q}[D_1,\dots,D_7]}{
\left(
\begin{array}{l}
D_1-D_5,\; D_2-D_6+D_7,\; D_3-D_7,\; D_4+D_5-D_7,\; 
D_1D_5,\; D_2D_6,\; D_3D_4D_7
\end{array}
\right)}$\\[2mm]
 & $\mathrm{Eff}(X)$ &
$\operatorname{Cone}\big(D_5,\; D_2,\; D_4,\; D_5-D_6+3D_7\big)$\\
 & $D$ &
$\{\,D_5,\; D_4\,\}$\\
\addlinespace[3mm]
\midrule
133 & $A^*(\mathbb{P}_\Delta)$ &
$\displaystyle
\frac{\mathbb{Q}[D_1,\dots,D_7]}{
\left(
\begin{array}{l}
D_1-D_5,\; D_2-D_6,\; D_3-D_7,\; D_4+D_5-D_7,\;
D_1D_5,\; D_2D_6,\; D_3D_4D_7
\end{array}
\right)}$\\[2mm]
 & $\mathrm{Eff}(X)$ & 
 $\operatorname{Cone}\big(D_5-D_6+3D_7,\;
 D_4,\; D_5,\; D_6\big)$\\
 & $D$ & 
 $\{\,D_6,\; D_5-D_6+3D_7\,\}$\\

\bottomrule
\end{longtable}
}


\bibliographystyle{plain}
\bibliography{ref.bib}

\end{document}